\UseRawInputEncoding
\documentclass[11pt]{amsart}
\usepackage[UKenglish]{isodate}
\usepackage{amsmath}
\usepackage{amsthm}
\usepackage{amsbsy}
\usepackage{amssymb}
\usepackage{amsfonts}
\usepackage[dvips]{lscape}
\usepackage{xcolor}
\usepackage{amscd}
\usepackage[all,cmtip]{xy}
\usepackage{euscript}
\usepackage{caption}  
\usepackage{parskip}
\usepackage{enumerate}
\usepackage{yfonts}
\usepackage{xcolor}
\usepackage{graphicx}
\usepackage{fancyhdr}
\usepackage{tikz-cd}
\usepackage{colortbl}
\usepackage{comment}
\usepackage{fullpage}
\usepackage{verbatim}
\usepackage{subcaption}
\usepackage{comment}
\usepackage{graphicx}
\usepackage{caption}
\usepackage{subcaption}
\usetikzlibrary{positioning}
\usetikzlibrary{decorations.pathreplacing}

\usepackage[margin=1in]{geometry}
\usepackage[expansion=false]{microtype}

\theoremstyle{plain}

\theoremstyle{plain}
\newtheorem{theorem}{Theorem}[section]

\newtheorem{proposition}[theorem]{Proposition}
\newtheorem{lemma}[theorem]{Lemma}
\newtheorem{corollary}[theorem]{Corollary}

\newtheorem{conjecture}[theorem]{Conjecture}

\newcommand{\res}{\text{Res}}

\theoremstyle{remark}
\newtheorem{remark}[equation]{Remark}

\theoremstyle{definition}
\newtheorem{definition}[theorem]{Definition}

\renewcommand{\Im}{\textup{Im}}
\renewcommand{\Re}{\textup{Re}}

\newcommand{\vol}{\textup{vol}}

\newcommand{\Tr}{\overline{\normalfont \textrm{Tr}}}
\newif\iffinalrun
\iffinalrun
\else
 \fi

\iffinalrun
  \newcommand{\need}[1]{}
  \newcommand{\mar}[1]{}
\else
  \newcommand{\need}[1]{{\tiny *** #1}}
  \newcommand{\mar}[1]{\marginpar{\raggedright\tiny Fix Me:  #1 }}\fi

\usepackage{amssymb,amsmath,amsfonts,amsthm,epsfig,amscd,stmaryrd}
\usepackage{stmaryrd}

\usetikzlibrary{arrows}

\pgfdeclarelayer{background}
\pgfsetlayers{background,main}

\pgfmathsetmacro{\myxlow}{-2}
\pgfmathsetmacro{\myxhigh}{2}
\pgfmathsetmacro{\myiterations}{6}

\title
{A quantitative converse of  Fekete's  theorem}

\author{Bryce Joseph  Orloski and Naser Talebizadeh Sardari}

\address{Penn State department of Mathematics, McAllister Building, Pollock Rd, State College, PA 16802 USA}
\email{nzt5208@psu.edu,bjo5149@psu.edu}
\thanks{We would like to thank Professors  Sarnak, Zargar, and Zikatanov for their comments on an earlier version of this work. }

\begin{document}
\maketitle
{\centering\footnotesize To Peter Sarnak on the occasion of his seventieth birthday.\par}

\begin{abstract}
Given a compact subset $\Sigma \subset \mathbb{R}$ (or $\mathbb{C}$) with logarithmic capacity greater than zero,
we construct an explicit family of probability measures supported on $\Sigma$  such that their closure is all the  possible weak  limit measures of complete sets of conjugate algebraic integers lying inside $\Sigma$. We give an asymptotic formula for the number of algebraic integers with given degree and prescribed distribution.
We exploit the algorithmic nature of our approach
to give a family of upper bounds that converges to the smallest limiting trace-to-degree ratio of totally positive algebraic integers and improve the best previously known upper bound on the Schur-Siegel-Smyth trace problem to  1.8216.

\end{abstract}

\section{Introduction}\label{intro}
\subsection{Fekete's theorem} 
A celebrated theorem of Fekete states that if a fixed compact set  $\Sigma\subset \mathbb{C}$ contains an infinite collection of complete sets of conjugate algebraic integers then the logarithmic capacity of $\Sigma$ is greater than or equal to zero.
Fekete and Szeg\"o~\cite{MR72941} proved a partial converse to this theorem.
Smyth showed that for any given probability distribution $\mu$ which is the weak limit of a sequence of distinct uniform probability measures on the complex roots of irreducible monic integral polynomials all of whose roots lie in $\Sigma$, then
\begin{equation}\label{smyth}
 \int \log|Q(x)|d\mu(x)\geq 0    
\end{equation}
for every non-zero $Q(x)\in \mathbb{Z}[x]$. Serre~\cite[Lemma 1.3.7]{MR4093205} showed more generally that
\begin{equation}\label{Serresmyth}
 \int \log|Q(x,y)|d\mu(x)d\nu(y)\geq 0    
\end{equation}
for every two variable integral polynomial $Q(x,y)\in \mathbb{Z}[x,y]$ and every pair of limiting measures $\mu,\nu$. Serre~\cite{MR4093205} also commented that the $n$-dimensional version of~\eqref{Serresmyth}  holds  and in~\cite[Corollary 1.3.10]{MR4093205} considered the special case when $Q(x,y)=x-y$ in \eqref{Serresmyth} which implies
\[
I(\mu,\nu):=\int \log|x-y|d\mu(x)d\nu(y)\geq 0. 
\]
When $\mu=\nu$, let $I(\mu):=\int \log|x-y|d\mu(x)d\mu(y)$ be the logarithmic energy of $\mu$.
\\

In a recent breakthrough, Smith~\cite{Smith} proved that the inequalities in \eqref{smyth} are sufficient for $\mu$ to be a limiting measure when  $\Sigma\subset\mathbb{R}$.
The authors~\cite{OT} recently established this result for any symmetric compact subset $\Sigma$ of the complex plane with a new proof of Smith's result. 
In doing so, we found new applications of the following inequality  in~\cite{OT}~and \cite{lowerbound}
\begin{equation}\label{multi}
\int \log(|Q(x_1,\dots,x_n)|)d\mu(x_1)\dots d\mu(x_n)\geq 0
\end{equation}
for  every limiting  measure $\mu$ and every non-zero $Q(x_1,\dots,x_n)\in \mathbb{Z}[x_1,\dots,x_n]$. 
\\

Suppose that $\Sigma \subset \mathbb{R}$ (or $\mathbb{C}$ with the assumption that it is invariant under complex conjugation) is a compact subset  with logarithmic capacity greater than zero and let $\mu$ be  a probability measure  supported on $\Sigma\subset \mathbb{C}$ that satisfies ~\eqref{multi}. In this paper,  we prove two basic results about the set of algebraic integers inside $\Sigma$ that are equidistributed according to $\mu$. 
\\

 First, we prove in Theorem~\ref{main2} that the number of algebraic integers of a given degree $n$ and approximate distribution $\mu$ is about 
\[
e^{\frac{n^2}{2}I(\mu)+o(n^{2-\delta})}.
\]
 Theorem~\ref{main2} gives a purely arithmetic definition of the logarithmic energy of $\mu$ in terms of the growth of the number of algebraic integers modelling $\mu$ as a function of their degree.  Our proof of Theorem~\ref{main2} is based on 
the geometry of numbers and a discrete version of John's theorem~\cite[Theorem 1.6]{Tao} for convex bodies inside Euclidean spaces. We also use~\eqref{multi} in order to bound Minkowski's  successive minima.  
\\

Second,  we construct an explicit family of probability measures supported on small neighborhoods of $\Sigma$ in Theorem~\ref{approx} that satisfy~\eqref{multi} and are dense with respect to the weak topology among  all possible probability measures that satisfy~\eqref{multi}. Our construction in Theorem~\ref{approx} uses the counting formula in Theorem~\ref{main2}. As an application, we give a family of upper bounds in Corollary~\ref{Acor} that converges to the smallest limiting trace-to-degree ratio of totally positive algebraic integers and improve the previously best-known upper bound on the Schur-Siegel-Smyth trace problem to 1.8216 in Corollary~\ref{our_bound}.
 Theorem~\ref{approx} addresses the following  question of  Sarnak:
\begin{center}    
``What are  the measures with compact support $\Sigma$ that satisfy the infinite number of inequalities in~\eqref{smyth} or equivalently in~\eqref{multi}?''
\end{center}

Before this work, the only known measures satisfying \eqref{multi} were the equilibrium measure of $\Sigma$ and an explicit  construction by Serre~\cite{MR2428512} when $\Sigma$ is a specific closed interval discussed in further detail in the next subsection.

\subsection{The Schur-Siegel-Smyth trace problem.} In this subsection, we discuss the application of our main theorems to the Schur-Siegel-Smyth trace problem. 
For an algebraic integer $\alpha$ with $\deg(\alpha)=n$ and Galois conjugates  $\alpha_1,\dots,\alpha_n$, define 
\[\text{Tr}(\alpha):=\sum_{i=1}^{n}\alpha_i.
\]
We call $\alpha$ totally positive if all of its conjugates $\alpha_1,\dots, \alpha_n$ are real and non-negative. Let $\mathcal{A}$ denote the set of totally positive algebraic integers. 
Note that for a given $m\in\mathbb N$, there are only finitely many $\alpha\in \mathcal{A}$ of a given degree such that $\text{Tr}(\alpha)=m$.
The Schur-Siegel-Smyth trace problem asks for the value of the lowest limiting trace-to-degree ratio. More specifically,
define $$\Tr(\alpha) = \frac{\text{Tr}(\alpha)}{\deg(\alpha)}.$$
The Schur-Siegel-Smyth trace problem seeks the value of
$$\lambda^{SSS} = \liminf_{\alpha\in\mathcal A}\Tr(\alpha).$$
Much work has been done computing lower bounds of $\lambda$ initiated by Schur~\cite{Schur},  Siegel~\cite{MR12092}, and Smyth~\cite{MR0762691} who gave a lower bound on the trace  problem by minimizing $\int x d\mu(x)$ subject to \eqref{smyth}. In~\cite{lowerbound}, the authors used~\eqref{multi} to improve the best known lower bound on the Schur-Siegel-Smyth trace problem to \(
1.79824\leq \lambda^{SSS}.
\) For providing upper bounds, one needs to construct probability measures with small trace that satisfy~\eqref{smyth} which is the main goal of this paper. 

\subsubsection{History of Upper Bounds }
The first known upper bound on $\lambda^{SSS}$ was given by Schur~\cite{Schur} who proved that $\lambda^{SSS}\le 2$ by giving an infinite family of totally positive algebraic integers of average conjugate value exactly 2.
Later, Siegel~\cite{MR12092} gave an example of an infinite family of totally positive algebraic integers $\{p_n\}$ with $p_n$ being degree $n$ and the trace of $p_n$ being $2n-1$. For a long time, 2 was the only known upper bound for $\lambda^{SSS}$. For this reason, the problem was usually phrased as finding the smallest $\epsilon > 0$ so that there are only finitely many totally positive algebraic integers $\alpha$ for which $\Tr(\alpha)<2-\epsilon$. Work by Serre~\cite{MR2428512} and Smith~\cite{Smith} showed that $\lambda^{SSS}<1.898304$. We discuss this upper bound in the following details.
\newline

Let 
$\mu$ be a measure and define
\begin{equation}\label{logpot}
U_{\mu}(x):=\int \log|x-y|d\mu(y)
\end{equation}
to be the (logarithmic) potential of $\mu$. Note that 
\[
I(\mu)= \int U_{\mu}(x)d\mu(x),
\]
and for two measures $\mu$ and $\nu$, we define
\[
I(\mu,\nu):=\int U_{\mu}(x)d\nu(x)=\int U_{\nu}(x)d\mu(x).
\]

In his letter to Smyth in February 1998, Serre showed that if $Q$ is a polynomial with real coefficients with leading and constant coefficient of modulus at least 1, and if $y,z>0$, $c\in \mathbb{R}$ are constants such that 
\begin{equation}\label{Serre}
   x-z\log(x)-y\log|Q(x)|\geq c 
\end{equation}
for all $x>0$, then $c\leq 1.898303$.   In his second letter to Smyth in March 1998, Serre constructed a  measure $\nu_S$ supported on an  interval $[a,b]\subset\mathbb{R}^+$ such that
\begin{equation}\label{Serredual}
   x\geq c+\lambda_x\log|x|+\lambda_{\nu_S}U_{\nu_S}(x) 
\end{equation}
\[
U_{\nu_S}(0)=\int \log(x)d\nu_S(x)\geq 0
\]
for all $x>0$  and some $c>1.898302$, $\lambda_x,\lambda_{\nu_S}\geq 0.$  By taking roots of $Q$ according to the distribution $\nu_S$, it follows that for every $c<1.898302$, there exists some $Q$ and   $y,z>0$ such that \eqref{Serre}
 holds. Let $c_{S}:=\sup c$,  where $c$ satisfies \eqref{Serredual} for some choice of $\nu_S$, $\lambda_x,\lambda_{\nu_S}\geq 0$. The above implies that $c_S\in (1.898302,1.898303)$. 
 \\

Using Serre's construction, Smith showed that there exists another measure $\mu_S$ supported on the same interval $[a,b]$ satisfying the conditions
\begin{equation}\label{serrecond}
  U_{\mu_S}(x)\geq \gamma \log|x|  
\end{equation}
and
\begin{equation}\label{serrecond2}
U_{\mu_S}(0)=\int \log(x)d\mu(x)\geq 0    
\end{equation}
with
\[
\int x d\mu_S(x)\leq 1.898304
\]
for a fixed $\gamma\geq 0$.  Note  that $\mu_S$ satisfies  Smyth's necessary conditions~\eqref{smyth} for $Q(x)=x$ by definition and for every $Q(x)$ coprime to $x$, we have 
\begin{equation}\label{SerreSmyth}
 \int \log|Q(x)|d\mu_S(x)=\sum_{Q(\alpha)=0}U_{\mu_S}(\alpha)\geq \gamma \sum_{Q(\alpha)=0}\log |\alpha|\geq 0.   
\end{equation}
By Smith's theorem, we have
\[
\lambda^{SSS}\leq 1.898304.
\]
This was the best-known bound on $\lambda_{SSS}$ prior to this work. Let $C_s=\inf \int x d\mu(x)$ where $\mu$ satisfies \eqref{serrecond} and \eqref{serrecond2} for some $\gamma>0$. Next, we show that  $c_c\leq C_s$ which is known as weak duality in the context of linear programming.
In fact, suppose that a given probability measure $\mu$ supported on $\mathbb{R}^+$  satisfies the inequalities~\eqref{serrecond} and \eqref{serrecond2}. We take the average of \eqref{Serredual} with respect to $\mu$ and obtain
\begin{equation}\label{oplower}
\int x d\mu(x) \geq c+\lambda_x\int\log|x|d\mu(x)+\lambda_{\nu}\int U_{\nu}(x)d\mu(x).    
\end{equation}
By inequality~\eqref{serrecond2} and $\lambda_x\geq 0$, we have
\[
\lambda_x\int\log|x|d\mu(x)\geq 0.
\]
Moreover, by \eqref{serrecond},
\[
\int U_{\nu}(x)d\mu(x)=\int U_{\mu}(x)d\nu(x) \geq \int \gamma \log(x)d\nu(x)\geq 0.
\]
By \eqref{oplower} and the above inequalities, it follows that
\[
C_S \geq c_S.
\]

\subsection{Main results}

Before this work, it was not known whether  $c_s=C_S$ which is known as strong duality. As a consequence of Theorem~\ref{dual}, we show that $c_s=C_S$ and we express this optimal value in the following corollary. This answers a question that was raised by Serre in his letter~\cite{MR2428512} about finding the optimal value for the optimization problem in~\eqref{Serre}.
\begin{corollary}
Let  $c_S $ and $C_S$ be as above and  $(\hat x, \hat y)$ be the unique solution to
\[\begin{cases}
y\log y = (y-x)\log(y-x)+x\\
\log(x)\log(y)=\log(xy)\log(y-x).
\end{cases}\]
Then \[
c_S =C_S= \hat y - \hat y \log\hat y + \hat x - \hat x\log\hat x.\] 
\end{corollary}

In Theorem~\ref{approx}, we show that every limiting measure $\mu$ of uniform distributions on roots of monic integral polynomials is well-approximated by  probability measures $\eta$ for which there exists an integral polynomial $Q(x)\in \mathbb{Z}[x]$ such that
\[
U_{\eta}(x)\geq \gamma\log|Q(x)|
\]
for some constant $\gamma\geq 0$ and 
\[
\int \log|Q(x)|d\eta(x)\geq0.
\]
Therefore, Theorem~\ref{approx} reduces the upper bound problem to  measures $\eta$ satisfying the above conditions. 

\begin{definition}
    A Borel probability measure is {\it arithmetic} if
    $$\int\log|Q(x)|d\mu\ge 0$$
    for all $Q\in\mathbb Z[x]\setminus\{0\}$.
\end{definition}
\begin{definition}
    The {\it root distribution} of a polynomial $P$ is defined by $\mu_P:=\frac{1}{\deg(P)}\sum_{P(\alpha)=0} \delta_\alpha$ where $\delta_x$ is the Dirac-$\delta$ mass at $x$ and where the sum counts roots with multiplicity. For a monic polynomial $P$, we write 
    \[
    \Tr(P):=\int x d\mu_P(x)
    \]
\end{definition}
\begin{theorem}\label{approx}
    Suppose $\mu$ is a compactly supported arithmetic probability measure. For every $N\in\mathbb N$, there is some $n\ge N$ for which there exists a monic irreducible integral polynomial $p_n$ of degree $n$ and a compactly supported arithmetic probability measure $\mu_n$ which satisfy
    \[
    \int \log|p_n(x)|d\mu_n(x)\geq 0,
    \]
    \[
    U_{\mu_n}(x)\geq \gamma_n \frac{\log|p_n(x)|}{\deg(p_n)}
    \]
    for every $x\in \mathbb{C}$ and some constants $0<\gamma_n<1$ where $\lim_{n\to \infty}\gamma_n=1$, and 
    \[
    \lim_{n\to \infty} \mu_n=\mu
    \]
    in the weak-* topology.
    Furthermore, we can take $\mu_n$ so that $\forall n\in\mathbb N$, $\text{supp}(\mu_n)\supseteq\text{supp}(\mu_{n+1})$ and $\bigcap_{n\in\mathbb N}\text{supp}(\mu_n)=\Sigma$. If the support of $\mu$ lies inside $\mathbb R$, then $p_n$ and $\mu_n$ can be taken so that $p_n$ has all real roots and $\mu_n$ is also supported in $\mathbb R$.
\end{theorem}
In fact, we solve a more general optimization problem that was proposed  by Smith~\cite{Smith}.

\subsubsection{Primal problem}\label{primalintro}
Fix a finite subset of integer polynomials $A=\{Q_1,\dots,Q_k\}$. Let $\mathcal{M}^{+}$ be the space of all non-negative Borel measures with compact support and  finite mass  on the non-negative real numbers.
 Define 
\[
\mathcal{C}(A):=\{(\mu,b_Q) : \mu\in\mathcal{M}^+, b_Q\in\mathbb{R} \text{ for every } Q\in A\}.
\]
We call the following linear constraints for $(\mu,b_Q)\in \mathcal{C}(A)$ the \textit{primal constraints}
\begin{equation}\label{primal}
\begin{split}
      U_{\mu}(x)\geq \sum_{Q\in A}b_Q\log|Q(x)|,  
      \\
  \int \log|Q(x)|d\mu(x)\geq 0,
  \\
  \int d\mu(x)\geq 1,
  \\
  b_Q\geq 0 \text{ for every } Q\in A,
\end{split}
\end{equation}
where the first inequality holds for every $x\in \mathbb{C}$ and the second holds for every $Q\in A$. We say $(\mu,b_Q)\in \mathcal{C}(A)$ is a feasible point  if $(\mu,b_Q)$ satisfies all inequalities in~\eqref{primal}. We denote the subset of feasible points by $\mathcal{C}(A)^+$. Note that $(0,\dots,0)\in \mathcal{C}(A)$, but $(0,\dots,0)$ is not a feasible point. 
\\

The primal  problem is to minimize the trace $\int xd\mu(x)$ for $(\mu,b_Q)\in \mathcal{C}(A)^+$. Let
\begin{equation}\label{defA}
\Lambda_A:=\inf_{(\mu,b_Q)\in \mathcal{C}(A)^{+}}\int xd\mu(x).
\end{equation}
Suppose that $B\subset A$ and $(\mu,b_Q)\in \mathcal{C}(B)^+$. It is easy to check that $(\mu,b'_Q)\in \mathcal{C}(A)^+$, where $b'_Q:=b_Q$ if $Q\in B$ and $b'_Q:=0$ if $Q\in A\setminus B.$ This implies that
\[
\Lambda_A\leq \Lambda_B
\]
We say a set $A$ is a minimal primal set if 
\[
\Lambda_A<\Lambda_B
\]
for every  subset  $B\subset A$, where $B\neq A$. 
 We say there is a  solution $\mu$ to the primal problem associated to $A$ if there exists some coefficients $b_Q\geq 0$ such that $(\mu,b_Q)\in \mathcal{C}(A)^+$ and 
 \[
 \int xd\mu(x)=\Lambda_A.
 \]
Suppose that $(\mu,b_Q)$ is a solution to primal problem and that $b_Q=0$ for some $Q\in A$. It follows that $\Lambda_{A\setminus\{Q\}}=\Lambda_A$ which implies $ A$ is not a  minimal primal set.

\begin{theorem}\label{main1} 
    Fix a finite subset of integral polynomials $A\subset \mathbb{Z}[x]$ with all real roots. There is a  solution $\mu_A$ to the primal problem associated to $A$. Furthermore,  $\mu_A$ is supported on a finite union of intervals $\Sigma_A\subset [0,18]$ such that
    \[
    \begin{array}{cc}
          \int x d\mu_A(x)=\Lambda_{A},\\
          I(\mu_A,\nu_A)=0,\\
    U_{\mu}(x)\geq \sum_{Q\in A} b_Q \log|Q(x)|,
    \end{array}
    \]
where equality holds in the third line for every $x\in \Sigma$ with some scalars $b_Q\geq  0$,  and $\nu_A$ is any
 solution to the dual problem  in Theorem~\ref{dual}. Moreover, suppose that $\Sigma_A:=\bigcup_{i=0}^l[a_{2i},a_{2i+1}] $. For every gap interval $(a_{2i+1},a_{2i+2})$, there is some $Q\in A$ that vanishes inside $(a_{2i+1},a_{2i+2})$. The density function of $\mu_A$ is
$$\frac{|P(x)|}{\prod_{Q\in A}|Q(x)|\sqrt{|H(x)|}}$$
where $H(x)=\prod_{i=0}^{2l+1}(x-a_i)$ and  $P(x)$ is a polynomial of degree $\sum_{Q\in A}\deg(Q)+l$ if $\{x\}\subset A$. Lastly, if we assume that  for every $Q\in A$ all roots are non-negative and $\Tr(Q)< \Lambda_A$, then $I(\mu_A)=0$, and $\mu_A$ is the unique solution to the primal problem associated to $A$. 
\end{theorem}

\begin{corollary}\label{Acor} If $A\subset\mathbb Z[x]$ is finite, then $\lambda^{SSS}\le \Lambda_A$. Furthermore, for all $\epsilon>0$, there is a finite set $A$ of positive real rooted integer polynomials so that $\Lambda_A-\lambda^{SSS}<\epsilon$.
\end{corollary}
\begin{proof}
By~Theorem \ref{main1}, there exists a primal measure $\mu_A$ with $\int x d\mu_A(x)=\Lambda_{A}$, where $\mu_A$ satisfies the conditions \eqref{primal}. It is easy to check that $\mu_A$  satisfies Smyth's necessary conditions~\eqref{smyth} for every non-zero $Q(x)\in \mathbb Z[x]$ as we checked it for $A=\{x\}$ in \eqref{SerreSmyth}; see also \cite[Proposition~5.9]{Smith} for general $A$. Therefore, $\mu_A$ is the limiting measure of some sequence of conjugate algebraic integers which implies $\lambda^{SSS}\le \Lambda_A$. 
\\

Next, suppose that $\alpha_n$ is a sequence of totally positive algebraic integers such that
    \[
    \lim_{n\to \infty}\Tr(\alpha_n)=\lambda^{SSS}.
    \]
    Let 
    \[
    \mu_{\alpha_n}:=\frac{1}{\deg(\alpha_n)}\sum_{\alpha_{n,j}\in [\alpha_n]}\delta_{\alpha_{n,j}}
    \]
    be the uniform discrete probability measure on the set of Galois conjugates  of $\alpha_n$. Since $\mu_{\alpha_n}$ is a tight family of probability measures, Prokhorov's theorem implies that there exists a convergent sub-sequence. Suppose that 
\(
\mu
\)
is a limiting measure of a sub-sequence. By Lemma 2.6 of \cite{lowerbound}, it follows that $\mu$ is supported inside the interval $[0,18]$. By Theorem \ref{approx}, it follows that for every $\varepsilon>0$, there exists an integral polynomial $p$ such that
$\Lambda_A-\lambda^{SSS}<\epsilon$, where $A=\{p\}$. This completes the proof of our Corollary.
\end{proof}
\begin{remark}
Let $A_0\subset A_1\subset\dots$ be an increasing chain of finite subsets of integer polynomials so that $\displaystyle\lim_{n\to\infty} \Lambda_{A_n}\to \lambda^{SSS}$. We conjecture that $\lim \mu_{A_n}=\mu$ exists and is the unique arithmetic measure with $\int xd\mu(x)=\lambda^{SSS}$ and $I(\mu)=0$; see Conjecture~\ref{unique_opt_dist} and Conjecture~\ref{opt_dist_zero_energy} for further detail. 
\end{remark}
\subsubsection{Dual problem}\label{dualintro} Fix $A$ and let $\mathcal{M}^+$ be as in the previous subsection. 
We  construct the analogue of $\nu_S$ for every finite subset $A$ which we denote by $\nu_A$. Our construction identifies $\nu_A$ as a solution to a dual optimization problem.  We say $\nu\in \mathcal{M}^+$ is a feasible dual measure associated to $A$ if it satisfies the following
\begin{equation}\label{dualconstraint}
    \begin{split}
        \int \log|Q(x)|d\nu(x)\geq 0,
        \\
      x-U_{\nu}(x)-\sum_{Q\in A}\lambda_{Q}(\nu)\log|Q(x)| \geq \lambda_{0}(\nu) \geq  0 
    \end{split}
\end{equation}
for every $Q\in A$ in the first inequality and every  $x\in\mathbb{R}^+$ in the second for some non-negative constants $\lambda_Q(\nu)\geq 0$. We denote the above by the \textit{dual constraints}. 
The dual problem is to maximize $\lambda_0(\nu)$ among feasible dual measures associated to $A$. Let
\begin{equation}
\lambda_A:=\sup_{\nu} \lambda_0(\nu),  
\end{equation}
where $\nu$ is varying among feasible dual measure associated to $A.$
Note that $\lambda_0(0)=0$, which implies $\lambda_A\geq0$. Moreover, if for some $Q\in A$ all roots are non-negative and $\Tr(Q)< \lambda_0(\nu)$, then $\lambda_Q(\nu)\neq 0$; see Proposition~\ref{smooth} for a quantitative version of this.

\begin{theorem}\label{dual}
Fix  a finite subset of integral polynomials $A\subset \mathbb{Z}[x]$ with all real roots. There is a feasible dual measure  $\nu_A$  with $\lambda_0(\nu_A)=\lambda_A$, and we have \[
\lambda_A=\Lambda_A.
\]

    Let $\mu_A$ be any solution in Theorem~\ref{main1} with support $\Sigma_A$ and expectation $\int x d\mu_A(x)=\Lambda_{A}$. Moreover, $\nu_A$ has the same support as $\mu_A$ with
       \[
    \begin{array}{cc}
                   x\geq \Lambda_A +\sum_{Q\in A}\lambda_{Q} \log|Q(x)|+ U_{\nu_A}(x)\text{ and}\\I(\mu_A,\nu_A)=0
    \end{array}
    \]
where in the first line the inequality holds for non-negative $x$ and equality holds for every $x\in \Sigma_A$ with some scalars $\lambda_Q\geq 0$ and $\lambda_0> 0$. Moreover if $\Sigma_A:=\bigcup_{i=0}^l[a_{2i},a_{2i+1}] $ and $\{x\}\subset A$, then $\nu_A$ has the following density function 
\[
\frac{|p(x)|\sqrt{|H(x)|}}{\pi\prod_{\lambda_Q\neq 0}|Q(x)|},
\]
where $H(x)=\prod_{i=0}^{2l+1}(x-a_i)$ and $p(x)$ is a monic polynomial with degree $\sum_{Q\in A} \deg(Q)-l-1$. Furthermore, $I(\nu_A)< 0$, and 
\[
\Lambda_A \geq \lambda^{SSS}\geq \Lambda_A+\frac{I(\nu_A)}{2\nu_A(\mathbb{R})}.
\]
\end{theorem}
\begin{remark}
    Note that Theorem~\ref{dual} gives both upper bound and lower bound on $\lambda^{SSS}$. Our lower bound associated to $A$ proved in \cite[Theorem 1.1]{lowerbound} is greater than $\Lambda_A+\frac{I(\nu_A)}{2\nu_A(\mathbb{R})}$ ; see Conjecture~\ref{supp_containment} and \cite{lowerbound} for further detail. We observed numerically that $\frac{I(\nu_A)}{\nu_A(\mathbb{R})}\geq \frac{I(\nu_B)}{\nu_B(\mathbb{R})}$ for $B\subset A$ and conjectured that $\limsup_{A}\frac{I(\nu_A)}{\nu_A(\mathbb{R})}=0$; see Conjecture~\ref{Inu=0}.
\end{remark}

In the trivial case where $A=\emptyset$, we have the following corollary of our main theorem. 
\begin{corollary}\label{schur_bound}
    If $A=\emptyset$, then $\Lambda_A=2$. Moreover, the optimal primal measure has density function
\[d\mu_A = \frac{dx}{\pi\sqrt{(4-x)x}}\]
and the optimal dual measure has density function 
\[
d\nu_A=\frac{\sqrt{4-x}}{\pi\sqrt{x}}dx
\]
with its potential function being $x-2$ on $\Sigma_A=[0,4]$.
\end{corollary}
Corollary~\ref{schur_bound} recovers Schur's and Siegel's upper bound on $\lambda^{SSS}$. 

\begin{corollary}\label{serre_bound}
If $A=\{x\}$, then $\Lambda_A=1.898302\dots$. The optimal measure $\mu_A$ is supported on
$\Sigma=[a,b]$ where
\begin{align*}
    a &= 0.0873528949\dots\text{ and }\\
    b &= 4.411076350\dots.
\end{align*}
This measure $\mu_A$ satisfies $U_{\mu_A}(x)= t \log(x)$ on $[a,b]$ where $t=0.215485...$ and
$U_{\mu_A}(0)=0$.
The dual measure is simply
$d\nu_A=\frac{\sqrt{(x-a)(b-x)}}{x\pi}dx$.    
\end{corollary}
As mentioned earlier, this measure $\mu_{\{x\}}$ in Corollary \ref{serre_bound} was approximated numerically by Serre~\cite{MR2428512} which Smith~\cite{Smith} used to prove  the best previous upper bound $\lambda^{SSS} \le 1.898304.$ 

\subsection{Gradient descent}
We wrote and implemented a gradient descent algorithm to approximate $\mu_A$ and $\nu_A$ for fixed $A$. 
Once given such approximate densities, we can estimate $\Lambda_A$ by forcing it into a small interval which provides an effective upper bound for $\Lambda_A\ge \lambda^{SSS}$.

\begin{corollary}\label{first_numer_approx}
Let $A=\{x,x-1\}$. Then the support of $\mu_A$ and $\nu_A$ is $[a_0,a_1]\cup [a_2,a_3]$ where
\[
a_0\approx 0.0706597128759717,a_1\approx 0.7191192204214787, a_2\approx 1.337668148298108, \text{ and }a_3\approx 4.687953934364709
\]
with
\[
|\Lambda_A- 1.84701204|\le 10^{-7}.
\]
\end{corollary}

\begin{corollary}\label{approx_quad}
Let $A=\{x,1-x,x^2-3x+1\}$.
We have
\[
\Sigma_A\approx[0.0743803, 0.2966453 ]\cup [0.4803800, 0.708246] \cup [1.348128,2.333811] \cup [2.9387115,4.592219] 
\]
and 
\[
|\Lambda_A- 1.8224798|\le 10^{-7}.
\]
\end{corollary}

\begin{corollary}\label{our_bound}
Let $A=\{x,1-x,x^2-3x+1,x^3-5x^2+6x-1\}$. We have
\[
\begin{split}
\Sigma_A\approx [0.061291, 0.175150]\cup [0.222340, 0.31342]\cup[0.457553, 0.746458]
\\
\cup[1.30038, 1.478553]
\cup[1.62963, 2.391040]\cup[ 2.87144, 3.108899]\cup[3.387062, 4.905820] 
\end{split}\]
and 
\[
|\Lambda_A-1.8215998| \le 10^{-7}.
\]
\end{corollary}

\begin{figure}\label{opt_measure_apx}
\centering
\begin{minipage}{.5\textwidth}
  \centering
  \includegraphics[width=1\linewidth]{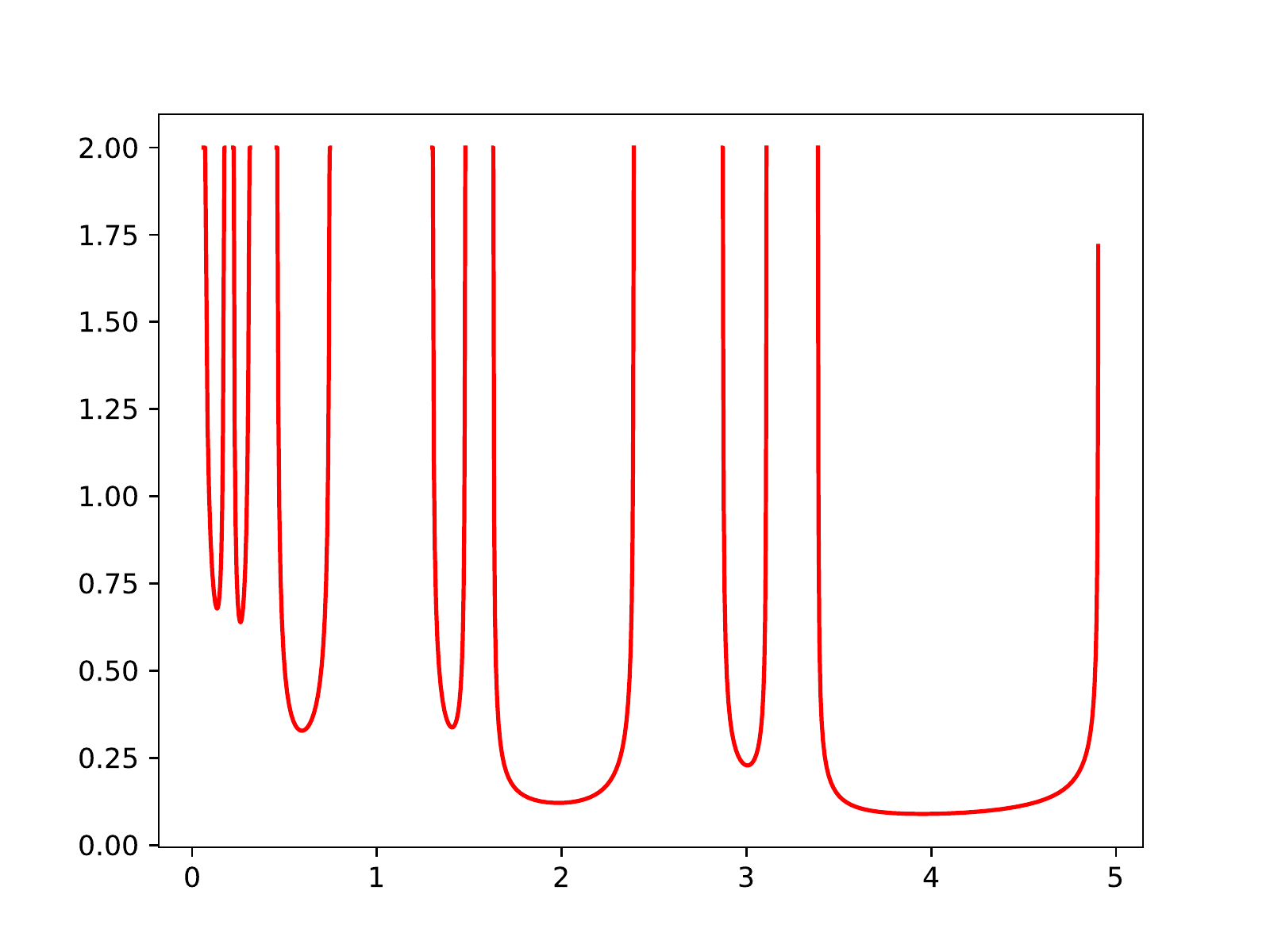}
  \captionof{figure}{Density function of $\mu_{\{x,x-1,x^2-3x+1,x^3-5x^2+6x-1\}}$ } 
        \label{primal_density_fig}
\end{minipage}%
\begin{minipage}{.5\textwidth}
  \centering
  \includegraphics[width=1\linewidth]{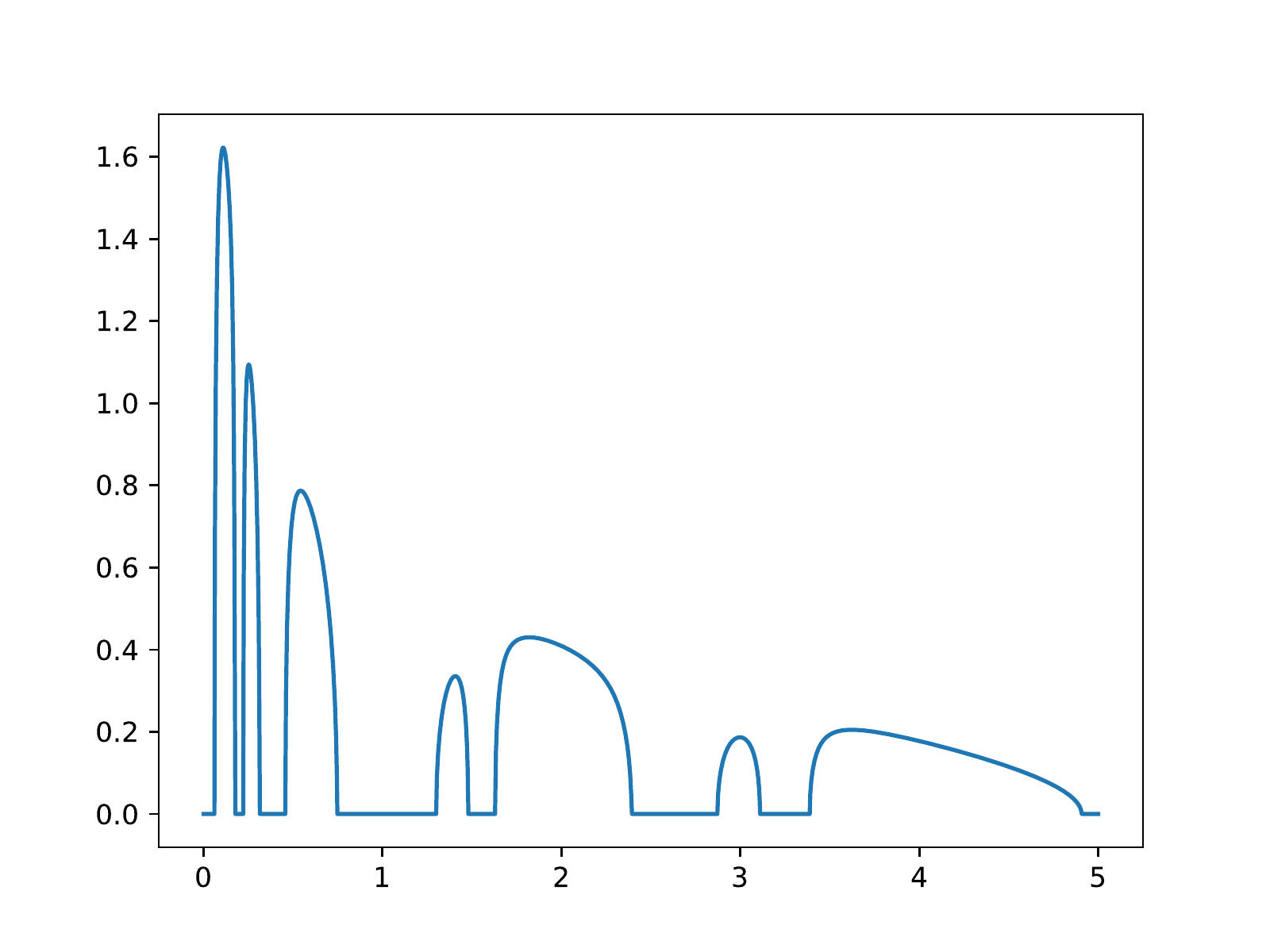}
  \captionof{figure}{Density function of $\nu_{\{x,x-1,x^2-3x+1,x^3-5x^2+6x-1\}}$ }
        \label{dual_density_fig}
\end{minipage}
\end{figure}

More details on the optimal primal and dual measures from these corollaries above are given in Figure~\eqref{primal_density_fig}, Figure~\eqref{dual_density_fig}, and Section \ref{our_bound_proof}.

\subsection{Counting algebraic integers of given degree and distribution}\label{result}
Suppose that $\Sigma\subset \mathbb{C}$ is compact and symmetric about the real axis with logarithmic capacity greater than or equal to zero. Let $\mathcal{P}_\Sigma$ be the space of probability measures supported on $\Sigma$ equipped with the weak-* topology. Let  
\[
\Sigma(\rho):=\{z\in \mathbb{C}: |z-\sigma|<\rho \text{ for some }\sigma\in\Sigma \}.
\]
If $\Sigma\subset \mathbb{R}$, let
\[
\Sigma_{\mathbb{R}}(\rho):=\{x\in \mathbb{R}: |x-\sigma|<\rho \text{ for some }\sigma\in\Sigma \}.
\]
Note that $\Sigma_{\mathbb{R}}(\rho)\subset \mathbb{R}$.
Next, we introduce some new notation to state our main theorem under some assumptions.  Let 
\[
[a,b]\times[c,d]:=\left\{z\in \mathbb{C}: \Re(z)\in [a,b] \text{ and } \Im(z)\in [c,d]\right\}\subset \mathbb{C}.
\]
We call $[a,b]\times[c,d]$ a rectangle, and we call it a real interval if $c=d=0$. We denote the boundary of a rectangle or an interval by $\partial R$ and $\partial I$.
A measure $\mu$ is a H\"older probability measure if there exists $\delta>0$ and $A>0$ such that
\[
\mu([a,b]\times[c,d]) \leq A\max(|b-a|, |d-c|)^{\delta}
\]
for every $a<b, c<d.$
Here $\delta$ is called a H\"older exponent of $\mu$.

\subsubsection{Generic symmetric sampling}
Suppose that $\Sigma \subset \mathbb{C}$ is a finite union of some rectangles and real intervals  and  $\Sigma$ is symmetric about the real axis.  We write
\[
\Sigma=\bigcup_{i=1}^{N_1}(R_i \cup \bar{R_i}) \bigcup_{j=1}^{N_2}I_j
\]
where $R_i\subset \mathbb{C}$ are rectangles, $\bar{R_i}$ is the reflection of $R_i$ by the $x$-axis and $I_j\subset \mathbb{R}$ are intervals. We define
\[
\Sigma_{\mathbb{R}}(\rho):= \bigcup_{i=1}^{N_1}(R_i(\rho) \cup \bar{R_i}(\rho)) \bigcup_{j=1}^{N_2}{I_j}_{\mathbb{R}}(\rho)
\]
for any $\rho>0$. Note that if $\bigcup_{j}{I_j}\neq \emptyset$, then $\Sigma_{\mathbb{R}}(\rho)\subsetneq \Sigma(\rho)$. We define the boundary of $\Sigma$ to be
\[
\partial \Sigma=\bigcup_{i=1}^{N_1}(\partial R_i \cup \partial\bar{R_i}) \bigcup_{j=1}^{N_2}\partial I_j
\]

 Suppose that $\mu$ is arithmetic and H\"older with support inside $\Sigma$. Given an integer $n$, we sample a collection of $n$ points according to the probability distribution  $\mu$ inside $\Sigma$ with some  properties that we introduce next.\\

Let $0<M\in \mathbb{Z}$ and $0<L \in \mathbb{Z}$ where 
\begin{equation}\label{samplcond}
n^{\varepsilon}<M \le n^{1/2-\varepsilon}, \text { and } L<n^{1-\varepsilon}M^{-2}  
\end{equation}
 for some fixed $\varepsilon>0.$
We partition each side of the rectangles and intervals inside $\Sigma^+$ into $M$ equal length sub-intervals. Namely, for a rectangle $R=[a,b]\times[c,d]$, we obtain a partition of $R$ by 
\[
R= \bigcup_{0\leq i,j< M}B_{ij},
\]
where $B_{ij}:=[a_i,a_{i+1}]\times[c_{j},c_{j+1}]$ and $a_i:=a+\frac{i(b-a)}{M}$ and $c_j:=c+\frac{j(d-c)}{M}.$
Similarly, for an interval $I=[a,b]\subset \mathbb{R}$, we have 
\[
I=\bigcup_{0\leq i< M}B_{i},
\]
where $B_{i}:=[a_i,a_{i+1}]$ and $a_i:=a+\frac{i(b-a)}{M}.$ 
\begin{definition}\label{genericsample} Fix $A>0$, and let $n,M,L$ be as above.  
   Let $\mathcal{Z}:=\{z_1,\dots,z_n \}\subset \Sigma$. We say that $\mathcal{Z}$ is a generic symmetric sampling according to $\mu$ with parameters $(n,M,L;A),$ if the following three conditions are satisfied.
   \begin{enumerate}
       \item $\mathcal{Z}$ is symmetric about the real axis.
       \item For every partition of rectangle $B_{ij}\subset R_l$, where $1\leq l\leq N_1$ and  $1\leq i,j \leq M$
   \[
   \left||\mathcal{Z} \cap B_{ij}| - \mu(B_{ij})n\right| \leq L,
   \]
   and for every partition of real interval $B_{i}\subset I_l$, where $1\leq l\leq N_2$ and $1\leq i \leq M$,
 \[
   \left||\mathcal{Z} \cap B_{i}| - \mu(B_{i})n\right| \leq L.
   \]
   
       \item $|z-w|\gg n^{-A}$ for every distinct pair $z,w\in \mathcal{Z}$, and any $z\in \mathcal{Z}$,  $w\in \bigcup_{0\leq i,j< M}\partial B_{ij} \bigcup_{0\leq i< M}\partial B_{i}.$
   \end{enumerate}
\end{definition}

\begin{theorem}\label{main2}
Suppose that $\Sigma \subset \mathbb{C}$ is compact and symmetric about the real axis and is a finite union of rectangles and real intervals. Furthermore, suppose that $\mu$ is arithmetic and H\"older with support inside $\Sigma$. Also let $\mathcal{Z}=\{\alpha_1,\dots,\alpha_n \}$ be sampled generically according to $\mu$ with parameters $(n,M,L;A).$ There  are asymptotically 
\[
e^{\frac{n^2}{2}I(\mu)+O(mn+m^2+n^{2-\delta})}
\]  
 irreducible polynomials $h$ of degree $n+m$ with complex roots $\{\beta_1(h),\dots, \beta_{n+m}(h) \}$ such that for every $1 \leq i\leq n$,
\[
|\beta_i(h)-\alpha_i|<d_{\Sigma}^{-m(1-O(n^{-\eta}))},
\]
where $m\geq n^{1-\delta}$ for some  $\delta>0$ and  $\eta>0$ only depend on $\mu$. Moreover, $\beta_i(h)\in \Sigma$ and
if $\alpha_i\in \mathbb{R}$ then $\beta_i(h)\in \mathbb{R}$
 for every $1 \leq i\leq n+m$.
\end{theorem}

\subsection{Applications to abelian varieties over finite fields}
By Honda Tate theory~\cite[Theorem 1.2 page 23]{serre_curves}, the isogeny class of simple abelian varieties over $\mathbb{F}_q$ and isogenies over $\mathbb{F}_q$ corresponds  to algebraic integers such that all their conjugates have norm $\sqrt{q}$. We apply our main theorem and Honda Tate theory to prove the following.
\begin{corollary}
    Let  $\mathbb{F}_q$ be a finite field and $\mu$ be a H\"older arithmetic probability measure on $[-2\sqrt{q},2\sqrt{q}]$.
    Suppose that $\mathcal{Z}=\{\alpha_1,\dots,\alpha_n \}$ is sampled generically according to $\mu$ with parameters $(n,M,L;A).$ There are  asymptotically
\[
e^{\frac{n^2}{2}I(\mu)+O(mn+m^2+n^{2-\delta})}
\]  isogeny classes of ordinary abelian varieties of degree $n+m$ with Frobenious eigenvalues $\{\beta_1,\dots,\beta_{m+n} \}$ such that $m$ is any integer $m \geq n^{1-\delta}$ and for every $1 \leq i\leq n$,
\[
|\alpha_i-\beta_i|<q^{-\frac{m}{2}(1+O(n^{-\eta}))},
\]
where  $\delta>0$ and  $\eta>0$ only depend on $\mu$. 
\end{corollary}

\subsection{Organization of the paper}
The asymptotic formula in Theorem~\ref{main2} is proved by proving an optimal upper bound in section~\ref{Upper bound} and an optimal lower bound
with the same exponent up to a power saving error term in section~\ref{lower bound}.  Our upper bound is proved by using a discrete version of John's theorem due to Tao and Vu~\cite[Theorem 1.6]{Tao} and our  bounds on the product of subsets of Minkowski's successive minima stated in~Proposition~\ref{lowerd}.
Our lower bound is proved following the method of proof of \cite[Theorem 1.5]{OT}. A highlight of our lower bound proof in this paper is that we do not use any results from potential theory.
This gives a new proof of our main theorems in \cite{OT} and Smith's main theorem in~\cite{Smith} without relying on the results of~\cite{logpotentials}.
Section~\ref{density} presents a proof of Theorem \ref{approx}.
Section~\ref{lin_prog} formulates a family of linear programming problems and proves that strong duality holds. Using this, Section~\ref{main_thms} proceeds to prove our main theorems on upper bounds to the Schur-Siegel-Smyth trace problem with an explanation on how this recovers Serre and Smith's bound in Section~\ref{serre_proof}.
Section~\ref{validation} validates the bounds given through numerical approximation, and finally, Section~\ref{conjectures} gives a list of open questions.

\section{Results from~\cite{OT}}\label{prior}
\subsection{Empirical logarithmic potential}
Suppose that $\Sigma \subset \mathbb{C}$ is a finite union of some rectangles and real intervals  and  $\Sigma$ is symmetric about the real axis.  Suppose that $\mu$ is arithmetic and H\"older with support inside $\Sigma$. Suppose that  $\mathcal{Z}\subset \Sigma$  is a generic symmetric sampling according to $\mu$ with parameters $(n,M,L;A),$ that is defined in~\eqref{genericsample}. Recall the logarithmic potential of $\mu$ from~\eqref{logpot} is  \[
  U_{\mu}(z):= \int \log|z-w| d\mu(w).
\]
We define the empirical logarithmic potential function $U_{\mathcal{Z}}$ as
\[
U_{\mathcal{Z}}(z):=\frac{1}{n}\sum_{\alpha\in\mathcal{Z}} \log|z-\alpha|.
\]

For each $z\in\mathbb C$, let $\hat z\in\mathcal Z$ be so that $|z-\hat z|=\min_{\alpha\in \mathcal{Z}}\log|z-\alpha|$. We define the  polynomial $p_{\mathcal{Z}}$ as
\begin{equation}\label{p_z}
p_{\mathcal{Z}}(x):=\prod_{\alpha\in\mathcal{Z}}(z-\alpha).
\end{equation}
Note that $U_{\mathcal{Z}}(z)=\frac{\log|p_{\mathcal Z}(z)|}{|\mathcal{Z}|}.$
The following proposition shows that $U_{\mathcal{Z}}(z)$ approximates  $U_{\mu}(z)$ with a power saving error in $n$. 

\begin{proposition}\label{emperical} Suppose that  $\mu$ is arithmetic with support inside $\Sigma$. Furthermore, suppose $\mu$ is H\"older with exponent $\delta>0$, and $\mathcal{Z}\subset \Sigma$  is a generic symmetric sampling according to $\mu$ with parameters $(n,M,L;A)$. 
We have 
\[
U_{\mathcal{Z}}(z)=U_{\mu}(z)+\min\left(\frac{\log|z-\hat{z}|}{n},0\right)+C\left(\frac{M^2L\log(n)}{n}+ M^{-\delta/2}\log(n)^{1/2}
\right)
\]
for some constant $C$ that only depends on $A$ and $\mu.$ Moreover, for $\Sigma \subset \mathbb R$, we have the stronger result
\[
U_{\mathcal{Z}}(z)=U_{\mu}(z)+\min\left(\frac{\log|z-\hat{z}|}{n},0\right)+C\left(\frac{ML\log(n)}{n}+ M^{-\delta/2}
\right).
\]
\end{proposition}
\begin{proof}
    It follows from  \cite[equation (33),(34)]{OT} in the proof of \cite[Proposition 2.8]{OT}.
\end{proof}

\subsection{Geometry of numbers}
Let $\Gamma_n$ be the lattice of integral polynomials of degree less than or equal to $n$. Suppose that $\mu$ is arithmetic and H\"older with support inside $\Sigma$,
and let
\begin{equation}\label{defKn}
K_n:=\left\{ p(x)\in \mathbb{R}[x]: \deg(p)\leq n, \frac{\log |p(z)|}{n} \leq U_{\mu}(z) \text{ for every } z\in \mathbb{C}\right\}.
\end{equation}
$K_n$ is a symmetric, convex region inside the vector space of polynomials with real coefficients and degree less than or equal to $n$. The successive minima of $K_n$ on $\Gamma_n$ are defined by setting the $k$-th successive minimum $\lambda_k$ to be the infimum of the numbers $\lambda$ such that $\lambda K_n$ contains $k+1$ linearly-independent vectors of $\Gamma_n$.
We have $0 < \lambda_0\leq \lambda_1 \leq \dots \leq \lambda_n <\infty$. Minkowski's second theorem states that 
\begin{equation}\label{minksec}
  \frac{2^{n+1}\vol(K_n)^{-1}}{(n+1)!} \leq  \lambda_0\lambda_1\dots \lambda_n\leq 2^{n+1}\vol(K_n)^{-1}.
\end{equation}
We denote a fixed associated successive minima basis of polynomials by
\begin{equation}\label{minimabasis}
  B_n:=\{p_0,\dots,p_n\}  
\end{equation}

We cite some propositions from~\cite{OT}.
\begin{proposition}[Proposition 2.10~\cite{OT}]\label{vol_K_n}
We have
\[
\prod_{i=0}^{n}\lambda_i^{-1} \gg 2^{-n}\vol(K_n) \gg n^{-cn^{2-\frac{\delta}{6}}}e^{\frac{n^2}{2}I(\mu)}
\]
for some constant $C,\delta>0$.
\end{proposition}

\begin{proposition}[Proposition 2.13~\cite{OT}]\label{lowerd} For every $d\leq n$, we
 have 
\[
\prod_{i=0}^{d} \lambda_i^{-1} \leq (d+1)! e^{\frac{(2n-d)(d+1)}{2}I(\mu)}.
\]
\end{proposition}
\begin{proposition}[Proposition 2.14~\cite{OT}]\label{up}
    Suppose that $\mu$ is arithmetic and H\"older with support inside $\Sigma$. We have
\[
\lambda_n \leq   n^{cn^{1-\delta}}. 
\]
for some positive constants $\delta,C>0$.
\end{proposition}

\begin{proposition}[Proposition 2.16~\cite{OT}]\label{chpolreal}
 Given a monic polynomial with real coefficients $p(x)$ of degree $n$, $\rho>0$, and an even integer $n^{\varepsilon}\leq m\leq n,$ there exists a monic polynomial with real coefficients  $Q_m(x)$ such that $\deg(Q_m)=m$ and each of the following hold.
    \begin{enumerate}
    \item The degree $i$  coefficient of the product $pQ_m$  is an even integer for  $n\leq i\leq n+m-1$.
    \item \label{2deform} Take $X$ to be the set of roots of $Q_m$. Then $X$ is a subset of $\Sigma_{\mathbb{R}}(\rho)$ and symmetric about the real axis. 
    \item \label{4deform}We have 
    \begin{equation*}
\frac{\log |Q_{m}(z)|}{m} =\frac{\min_{\alpha\in X}\log|z-\alpha|}{m}+  U_{\mu_{eq}}(z)+O(m^{-\delta})
\end{equation*}
for  all $z\in \mathbb{C}$ and some $\delta>0,$ where $\mu_{eq}$ is the equilibrium measure for $\Sigma_{\mathbb{R}}(\rho/10).$

    \item~\label{3deform} Given any root $\alpha$ of $Q_m$ and any $\alpha'$ that is either a root of $p$  or a boundary point of $\Sigma_{\mathbb{R}}(\rho)$, we have
\(
n^{-3} \leq |\alpha-\alpha'|.
\)
\end{enumerate}
\end{proposition}

\section{Counting Integral Polynomials}\label{counting}
\subsection{Upper Bound}\label{Upper bound}

In this section, we give an  upper bound on the number of monic, integral, irreducible polynomials $h$ satisfying the conditions of
Theorem~\ref{main2}. We recall those conditions.  Suppose that $\mathcal{Z}=\{\alpha_1,\dots,\alpha_n \}$ is sampled generically according to $\mu$ with parameters $(n,M,L;A).$ Suppose that $h$ is an
irreducible polynomial  of degree $n+m$ with complex roots $\{\beta_1(h),\dots, \beta_{n+m}(h) \}$ such that  for every $1 \leq i\leq n$,
\[
|\beta_i(h)-\alpha_i|<d_{\Sigma}^{-m(1+O(n^{-\eta}))},
\]
if $\alpha_i\in \mathbb{R}$ then $\beta_i(h)\in \mathbb{R}$, 
and $\beta_i(h)\in \Sigma$ for every $1 \leq i\leq n+m$ where $m\ge n^{1-\delta}$ for some $\delta>0$.
\begin{proposition}\label{upp}
The number of integral polynomials $h$ that satisfies the conditions of Theorem~\ref{main2} listed above is less than
\[
e^{\frac{n^2}{2}I(\mu)+O(mn+m^2+n^{2-\delta})}
\]  
for some positive constant $\delta>0$.
\end{proposition}
\begin{proof}
    By Lemma~\ref{hinK}, we know that $h\in e^{C(m+n^{1-\delta'})}K_{n+m} \cap \Gamma_{m+n}$. By Proposition~\ref{Tao}, we have
    \[
    \begin{split}
    |e^{C(m+n^{1-\delta'})}K_{n+m} \cap \Gamma_{m+n}| &\leq e^{C(n+m)(m+n^{1-\delta'})}(c(n+m))^{9(n+m)/2} e^{\frac{(m+n)^2}{2}I(\mu)}
    \\
    &\leq 
    e^{\frac{n^2}{2}I(\mu)+O(mn+m^2+n^{2-\delta})}
     \end{split}
    \]
    
for some positive constant $\delta>0$.
\end{proof}
 
\begin{lemma}\label{hinK}
Let $K_{n+m}$ and $h$ be as above. Then\[
h\in e^{C(m+n^{1-\delta'})}K_{n+m}
\]
for some constants  $C,\delta'>0$.
\end{lemma}
\begin{proof}
    Let $\mathcal{X}:=\{\beta_1(h),\dots, \beta_{n}(h)\}$. Since $|\beta_i(h)-\alpha_i|<d_{\Sigma}^{-m}$ and $d_{\Sigma}^{-m} \ll n^{-A}$, it is easy to check that $\mathcal{X}$ is also a generic sampling of $\mu $ with the same parameters $(n,M,L;A)$.
    By Proposition~\ref{emperical}, 
    \[
    \frac{\log|p_{\mathcal{X}}(z)|}{n}=U_{\mu}(z)+\min(\log|z-\hat{z}|,0)+C\left(\frac{M^2L\log(n)}{n}+ M^{-\delta/2}\log(n)^{1/2}
\right).
    \]
   Let $Q_m(z):=\prod_{i=1}^m(z-\beta_{n+i}(h))$. Note that $|z-\beta_{n+i}|\leq |z|+C$ for all $i=1,\dots m$ for some fixed constant $C$. Hence,
   \[
   \frac{\log|Q_m(z)|}{m} \leq \log|z|+C
   \]
   By \cite[Lemma 2.1 and Lemma 2.2]{OT},
   \[
   U_{\mu}(z)\geq \max(\log|z|,0)+C_1
   \]
for some constant $C_1$ that depends on $\mu$. Hence,
\[
\frac{\log|Q_m(z)|}{m} \leq U_{\mu}(z)+C_2
\]
for some constant $C_2$. 
Finally, we have
\[
\log|h(z)|= \log|p_{\mathcal{X}}(z)|+\log| Q_m(z)|\leq (m+n)U_{\mu}(z)+C'\left(m+M^2L\log(n)+ nM^{-\delta/2}\log(n)^{1/2}
\right)
\]
for some constant $C'$. The lemma follows from~\eqref{samplcond}, which implies 
\[
M^2L\log(n)+ nM^{-\delta/2}\log(n)^{1/2}=O(n^{1-\delta'})
\] 
for some $\delta'>0$.

\end{proof}

We apply a discrete version of John's theorem theorem due to Tao and Vu~\cite[Theorem 1.6]{Tao} to give an upper bound on $|\Gamma_n\cap K_n|$ in terms of the volume of $K_n$. For a linearly independent tuple of integral polynomials ${\bf{P}}:=(p_0,\dots,p_d)\in \Gamma_n^d$ and  ${\bf{N}}:=(N_0,\dots,N_d)\in \mathbb{Z}^d$, define
\[
{\bf{P}}({\bf{N}}):=\left\{\sum_{i=0}^d n_ip_i: -N_i\leq n_i\leq N_i \right\}\subset \Gamma_n
\]

\begin{proposition}\label{Tao}
    We have
   \[
 |\Gamma \cap K_n|< (cn)^{9n/2} e^{\frac{n^2}{2}I(\mu)}.
\]
    for some constant $c,\delta>0$.
\end{proposition}
\begin{proof}
By~\cite[Theorem 1.6]{Tao}, there exists a linearly independent tuple of integral polynomials ${\bf{P}}:=(p_0,\dots,p_d)\in \Gamma_n^d$ and  ${\bf{N}}:=(N_0,\dots,N_d)\in \mathbb{Z}^d$,  such that
\[
{\bf{P}}({\bf{N}}) \subset \Gamma \cap K_n
\]
and 
\[
|\Gamma \cap K_n|\leq (cn)^{7n/2} |{\bf{P}}({\bf{N}})|
\]
for some $c>0$. Note that
\[
|{\bf{P}}({\bf{N}})|=\prod_{i=0}^d(2N_i+1)
\]
Suppose that $N_0\geq N_1\geq \dots\geq N_d$ is a decreasing sequence. By definition of Minkowski's minima, since
\(
N_ip_i\in K_n,
\) we have
\(
\lambda_i^{-1}\geq N_i
\)
for every $0\leq i\leq d$. This implies that
\[
|{\bf{P}}({\bf{N}})|\leq \prod_{i=0}^d 3\lambda_i^{-1} \leq 3^{d+1}(d+1)! e^{\frac{(2n-d)(d+1)}{2}I(\mu)},
\]
where we used Proposition~\ref{lowerd} for the last inequality. Therefore,
\[
|\Gamma \cap K_n|<(cn)^{7n/2} 3^{d+1}(d+1)! e^{\frac{(2n-d)(d+1)}{2}I(\mu)}.
\]
 We note $d\leq n$ and the above inequality is maximized when $d=n$, which implies
 \[
 |\Gamma \cap K_n|< (cn)^{9n/2} e^{\frac{n^2}{2}I(\mu)}
 \]
 for some (potentially larger) $c$ dependent only on $\mu$.
\end{proof}

\subsection{Lower Bound}\label{lower bound}
In this section, we give a lower bound on the number of integral polynomials $h$ satisfying the conditions of Theorem~\ref{main2}. Our proof is similar to that of \cite[Theorem 1.5]{OT}.

\begin{proof}[Proof of Theorem~\ref{main2} (lower bound)]\label{rouche}
Suppose that $ \Sigma=\bigcup_{i}(R_i \cup \bar{R_i}) \bigcup_{j}I_j$ with $d_{\Sigma}>1$. By continuity of the transfinite diameter~\cite[Theorem J]{MR72941}, it is possible to find a subset $\Sigma_1\subset \Sigma$ which is a finite union of rectangles and intervals such that ${\Sigma_1}_{\mathbb{R}}(\rho)\subset \Sigma$ and $d_{\Sigma_1}> 1$ for some $\rho>0$.
\\

Let $\mathcal{Z}=\{\alpha_1,\dots,\alpha_n \}$ be as in the statement of Theorem~\ref{main2}. Let $p_{\mathcal{Z}}(x)$ be as in~\eqref{p_z}. By Proposition~\ref{emperical}, we have
\[
|p_{\mathcal{Z}}(z)| \geq |z-\hat{z}| e^{nU_\mu(z)}e^{-Cn^{1-\delta_0}}
\]
for some $\delta_0>0$ and $C>0$.
We apply Proposition~\ref{chpolreal} to $\Sigma_1$, $p_{\mathcal{Z}}(x)$ and $m:= n^{1-\delta}$ for some $\delta>0$ that we specify later to obtain $Q_m(x)$ such that 
      $p_{\mathcal Z}(x)Q_m(x)$ has $n+m$ distinct roots inside $\Sigma$ which are $n^{-3}$ apart from the boundary of $\Sigma$. Moreover, we have
      \begin{equation}\label{robin}
\left| p_{\mathcal Z}(z)Q_m(z) \right|\geq |z-\hat{z}| e^{nU_\mu(z)} d_{\Sigma_1}^m n^{-Cn^{1-\delta_0}},
\end{equation}
where we used the third and fourth properties of Proposition~\ref{chpolreal} and the fact that $U_{eq}(z)>d_{\Sigma_1}$.
Here $\hat{z}$ is the closest root of $p_\mathcal{Z}(z)Q_m(z)$ to $z$.
Moreover, 
\[
p_\mathcal Z(x)Q_m(x)=x^{m+n}+\sum_{i=0}^{m+n-1} a_ix^i
\]
where $a_i$ are even for $n\leq i\leq m+n-1.$ Let 
\[
r(x):=\frac{1}{2}+\sum_{i=0}^{n-1} \frac{a_i}{4} x^i.
\] We write $r(x)$ in terms of  the basis $B_{n-1}:=\{p_0,\dots,p_{n-1}\}$ obtained from Minkowski's successive minima for $K_{n-1}$ defined in~\eqref{minimabasis}, and obtain
\[
r(x)=\sum_{i=0}^{n-1} b_i p_i(x).
\] 
Let
\[
w(x):=\sum_{i=0}^{n-1} \lfloor b_i\rfloor p_i(x)
\]
and
\[
h_n(x):= x^{m+n}+\sum_{i=n}^{m+n-1} a_ix^i +4w(x)-2.
\]
By the Eisenstein criteria at the prime 2, $h_n(x)$ is irreducible. Moreover,
\[
h_n(x)=p_{\mathcal Z}Q_m(x)+\sum_{i=0}^{n-1-m} \beta_ip_i(x),
\]
where $|\beta_i|=4|\alpha_i-\lfloor\alpha_i\rfloor|<4$ for each $i$. By the definition of Minkowski's minima, we have that
\(
|p_i(z)|\leq \lambda_i e^{(n-1)U_\mu(z)}
\)
for every $z\in \mathbb{C}$. By Proposition~\ref{up},
\(\lambda_i\leq n^{cn^{1-\delta_1}}\)
for some $\delta_1>0$.
Take $m=n^{1-\delta}$ for some $\delta>0$ such that $\min(\delta_1,\delta_0)>\delta>0$. Let ${\bf{P}}:=(p_0,\dots,p_{n-1})$ and  ${\bf{N}}:=(N_0,\dots,N_{n-1})\in \mathbb{Z}^d$, where $N_i=\left\lfloor \frac{(d_{\Sigma_1})^{n^{1-\delta_2}}}{\lambda_i}\right\rfloor$ for some $\delta_1>\delta_2>\delta$. We
define
\[
\mathbf{H}_n(\mathbf{N}):=\left\{h_n+\sum_{i=0}^{n-1} 4n_ip_i: -\left\lfloor\frac{N_i}{4}\right\rfloor\leq n_i\leq \left\lfloor\frac{N_i}{4}\right\rfloor \right\}\subset \Gamma_n.
\]
We show that every polynomial inside $\mathbf{H}_n(\mathbf{N})$ satisfies the conditions of Theorem~\ref{main2}.
First, we give a lower bound on the number its elements.
By Proposition~\ref{vol_K_n}, we have
\[
|\mathbf{H}_n(\mathbf{N})| = \prod_{i} \left(2\left\lfloor\frac{N_i}{4}\right\rfloor+1\right) > \prod_i\left(\frac{N_i}{4}\right)\ge 4^{-n}d_{\Sigma_1}^{n^{2-\delta_2}}\prod \lambda_i^{-1}\geq e^{\frac{n^2}{2}I(\mu)+O(mn+m^2+n^{2-\delta})}
\]
for some $\delta>0$.  Next, we verify every element of $\mathbf{H}_n(\mathbf{N})$ satisfies the conditions of Theorem~\ref{main2}. 
Suppose that $h_n+\sum_{i=0}^{n-1} 4n_ip_i=p_{\mathcal Z}Q_m(z)+\sum_{i=0}^d (4n_i+\beta_i)p_i\in \mathbf{H}_n(\mathbf{N})$. It follows from~\eqref{robin} and our choices for $m$ and $N_i$ that
\begin{equation}\label{maininq}
\left| p_{\mathcal Z}Q_m(z) \right|\geq  |z-\hat{z}| e^{nU_\mu(z)} d_{\Sigma_1}^m n^{-Cn^{1-\delta_0}}>\left|\sum_{i=0}^{n-m-1} (\beta_i+4n_i)p_i(z)\right|
\end{equation}
for every $z$ such that $$|z-\hat{z}|>d_{\Sigma_1}^{-m(1-Cn^{-\delta'}\log(n))},$$
where $\delta'=\min(\delta_2,\delta_0)-\delta>0$ and $C>0$ is a positive constant. Let $\mathcal{X}$ be the root set of $Q_m$.
Suppose that $\alpha\in (\mathcal{Z}\cup \mathcal{X}) \cap \bigcup_{i}(R_i \cup \bar{R_i})$,
consider the disk 
$$D_{\alpha}:=\{z:
|z-\alpha|<d_{\Sigma}^{-m(1-O(n^{-\eta}))}\}$$
for some $\eta<\delta'$. Note that by our construction of $Q_m$ and the definition  of generic sampling~\ref{genericsample}, $p_{\mathcal{Z}}Q_m(z)$ has exactly the single root $\alpha$ inside $D_{\alpha}$. 
By~\eqref{maininq} and Rouch\'e's theorem, $h_n+\sum_{i=0}^{n-1} 4n_ip_i$ and $p_{\mathcal{Z}}Q_m(z)$ have exactly one root inside $D_{\alpha_i}$. 
\\

Similarly, suppose that $\gamma\in (\mathcal{Z}\cup \mathcal{X}) \cap (\bigcup_{j}I_j)$,
consider the real interval 
$$I_{\gamma}:=\{x\in \mathbb{R}:
|x-\gamma|<d_{\Sigma}^{-m(1-O(n^{-\eta}))}\}$$
for some $\eta<\delta'$. By our construction of $Q_m$ and the definition  of generic sampling~\ref{genericsample}, $p_{\mathcal{Z}}Q_m(z)$ has exactly the single real root $\gamma$ inside $I_{\gamma}$. 
By~\eqref{maininq} and intermediate value theorem, $h_n+\sum_{i=0}^{n-1} 4n_ip_i$ and $p_{\mathcal{Z}}Q_m(z)$ have exactly one root inside $I_{\gamma}$. Lastly, since $h_n$ is Eisenstein at the prime 2, so is every polynomial in $\mathbf{H}_n(\mathbf{N})$. This completes the proof of our Theorem.
\end{proof}

\section{Proof of Theorem~\ref{approx}}\label{density}
Our proof is based on the approximation results the authors proved in~\cite[Proposition 3.1]{OT}.
Our Theorem \ref{main2} counts algebraic integers whose roots distribute approximately to a given sampling.
We use this along with an approximation of potentials to prove the theorem.
In \cite{OT}, the authors make extensive use of the root distribution of a polynomial.

To find the polynomials desired in Theorem \ref{approx}, we use Theorem \ref{main2} to obtain polynomials with roots close to a given sampling of the desired measure. 
\begin{proof}[Proof of Theorem~\ref{approx}]
Define $\Sigma:=\text{supp}(\mu)$ and $\Sigma_n:=\overline{\Sigma_{\mathbb R}(n^{-1})}$ in the case that $\Sigma\subset\mathbb R$ and $\Sigma_n:=\overline{\Sigma(n^{-1})}$ otherwise.
We start by assuming without loss of generality that $\mu$ is H\"older. If it is not H\"older, then by Proposition 2.12 of \cite{OT} and Proposition 2.4 of \cite{Smith}, there is a sequence of arithmetic H\"older measures converging to $\mu$ on which we can use a diagonal argument.
Define $\tilde\mu_n:=(1-n^{-1})\mu+n^{-1}\mu_{eq,\Sigma_n}$ so that $I_{\tilde\mu_n}>0$, $\tilde\mu_n$ is H\"older, and $\tilde\mu_n$ converges weakly to $\mu$.
Furthermore, Disc$(\mu,\tilde\mu_n)\ll n^{-1}$.
Let $\min(\frac{1}{2},\delta)\le\delta_n\le 1$ be a H\"older exponent of $\tilde\mu_n$ where $\delta$ is a fixed H\"older exponent of $\mu$.
Let $\mathcal{Z}_{m,n}=\{z_{1,n},\dots z_{m,n}\}$ be a generic sampling of $\tilde \mu_n$ with parameters $(m,\lfloor m^{1/3}\rfloor,\lfloor m^{\frac{1}{3}-\frac{\delta_n}{6}}\rfloor; \frac{-4}{3})$ for $\Sigma\subset \mathbb R$ and $(m,\lfloor m^{1/3}\rfloor,\lfloor m^{\frac{1}{3}-\frac{\delta_n}{6}}\rfloor; \frac{-5}{6})$ otherwise as defined in definition \ref{genericsample}.
\newline

First, assume $\Sigma\subset\mathbb R$. By Theorem \ref{main2}, for sufficiently large $m$, we can find integral $Q_{m,n}$ with real roots $r_1,\dots, r_{\deg(Q_{m,n})}$ such that $\deg(Q_{m,n})=m+\Theta(m^{1-\epsilon})$ and $m$ of the roots (without loss of generality $r_1,\dots, r_m$) satisfy $|z_{i,n}
-r_i|\le d_{\Sigma_n}^{-\tilde C_n\sqrt{m}}$ for some constant $\tilde C_n>0$ and $\epsilon>0$.
We now estimate the potential of the root distribution of $Q_{m,n}$ using Proposition \ref{emperical}.
This states that
\begin{equation}\label{Zpot}
    U_{\mathcal Z_{m,n}}(x):=\frac{1}{m}\sum_{i=1}^m\log|x-z_{i,n}| = U_{\tilde\mu_n}(x)+\min\left(\frac{\log|x-\hat{x}|}{m},0\right)+O\left(m^{-\delta_n/6}\right)
\end{equation}
where $\hat{x}\in\mathcal{Z}_{m,n}$ minimizes $|x-z|$ for $z\in \mathcal{Z}_{m,n}$.
The proximity of the roots of $Q_{m,n}$ to $\mathcal{Z}_{m,n}$ gives
\begin{equation}\label{compare}
    U_{\mathcal{Z}_{m,n}}(x)-U_{\mu_{Q_{m,n}}}(x)= \frac{1}{\deg(Q_{m,n})}\left(\sum_{1\le i\le m}\log\left|\frac{x-z_{i,n}}{x-r_i}\right|-\sum_{i>m}\log|x-r_i|\right)+\frac{\deg(Q_{m,n})-m}{m\deg(Q_{m,n})}U_{\mathcal Z_{m,n}}(x).
\end{equation}
Using equations \eqref{Zpot} and \eqref{compare}, we get
\begin{equation}\label{poly_pot}
U_{\mu_{Q_{m,n}}}(x) = U_{\tilde\mu_n}(x) + \min\left(\frac{1}{\deg(Q_{m,n})}
\left(\log|x-\tilde{x}|+\sum_{i>m}\log|x-r_i|\right),0\right) + O\left(m^{-\delta_n/6}\right)
\end{equation}
where $\tilde{x}$ is the root of $Q_{m,n}$ closest to $x$.
\newline

We now compute $\mu_n$ so that for each large enough $n$, there is a sufficiently large $m_n$ for which $U_{\mu_n}(x)\ge \gamma_n\frac{\log|Q_{m_n,n}(x)|}{\deg(Q_{m_n,n})}$. For this, we give a proof similar to that of Proposition 3.1 of \cite{OT} as mentioned before. Define
$\mu_{m,n}$ to be the measure supported on intervals of length $m^{-2}$ around the elements of $\mathcal Z_{m,n}$, more precisely $\{x\in \mathbb R:\exists z\in \mathcal Z_{m,n},|x-z|<\frac{1}{2m^2}\}$. On each interval, the measure is given by $m$ times the Lebesgue measure so that $\mu_{m,n}$ is a probability measure.
\newline

Note that $U_{\mu_{m,n}}(x)=m\sum_{z\in\mathcal{Z}_{m,n}}\int_\frac{-1}{2m^2}^\frac{1}{2m^2}\log|x-z-y|dy.$ For fixed $x$, pick $z\in\mathcal{Z}_{m,n}\setminus\{\hat x\}$.
Then since $|x-z|\gg m^{-4/3}$, we have that
$$\left|m\int_\frac{-1}{2m^2}^\frac{1}{2m^2}\log|x-z-y|dy-\frac{1}{m}\log|x-z|\right|=m\left|\int_\frac{-1}{2m^2}^\frac{1}{2m^2}\log\left|1-\frac{y}{x-z}\right|dy\right|\ll m^{-5/3}.$$
Therefore,
$$U_{\mu_{m,n}}(x)=m\int_\frac{-1}{2m^2}^\frac{1}{2m^2}\log|x-\hat x-y|dy +\frac{1}{m}\sum_{z\in\mathcal{Z}_{m,n}\setminus\{\tilde x\}}\log|x-z|+O(m^{-2/3}).$$
Thus if $|x|\ge 2\sup_{y\in\text{supp}(\mu)}|y|$, then $U_{\mu_{m,n}}(x)=U_{\tilde\mu_n}(x)+O\left(m^{-\delta_n/6}\right)$. Otherwise,
since we have that $m\int_\frac{-1}{2m^2}^\frac{1}{2m^2}\log|x-\tilde x-y|dy\ll m^{-1}\log(m)$, then
$$U_{\mu_{m,n}}(x)=U_{\mathcal{Z}_{m,n}}(x)-\frac{1}{m}\log|x-\hat{x}|+O(m^{-2/3}).$$
Thus we have 
$U_{\mu_{m,n}}(x)=U_{\tilde\mu_n}(x)+O\left(m^{-\delta_n/6}\right)$
for all $x\in\mathbb R$ and hence by Theorem \ref{dom} for all $x\in \mathbb C$. A similar computation is done in the proof of Theorem 3.1 of \cite{OT} for the case $\Sigma\not\subset\mathbb R$ (in which $\mu_{m,n}$ is the union of Lebesgue measures on intervals of length $m^{-1}$ around points of $\mathcal Z_{m,n}$). Here it is computed that $U_{\mu_{m,n}}(x) = U_{\tilde\mu_n}(x) + O(m^{-1/6}+m^{-\delta_n/6}\log(m)^{1/2})$. In either case,
\begin{equation}\label{approx_pot}
    U_{\mu_{m,n}}(x) = U_{\tilde\mu_n}(x)+O\left
    (m^{-\delta_n/7}\right)
\end{equation}
\newline

Now fix $0<\gamma_n<1$ so that $\lim_{n\to\infty}\gamma_n=1$. Let $\nu_n$ be the equilibrium measure on $\Sigma_n$
and set $\mu_n:=\gamma_n\mu_{m_n,n}+(1-\gamma_n)\nu_n$ where $m_n$ is chosen later.
We see that
\begin{align*}
    U_{\mu_n}(x) &= \gamma_n U_{\mu_{m_n,n}}(x)+(1-\gamma_n)U_{\nu_n}(x)\\
    &= \gamma_n U_{\tilde\mu_n}(x)+(1-\gamma_n)U_{\nu_n}(x) + O\left(m_n^{-\delta_n/7}\right)&\text{by }\eqref{approx_pot}\\
    &=\gamma_n U_{\mu_{Q_{m_n,n}}}(x) - \min\left(\frac{\gamma_n}{\deg(Q_{m_n,n})}
\left(\log|x-\tilde{x}|+\sum_{i>m_n}\log|x-r_i|\right),0\right)\\ &\qquad + (1-\gamma_n)U_{\nu_n}(x) + O\left(m_n^{-\delta_n/7}\right) &\text{by }\eqref{poly_pot}.
\end{align*}
This works regardless of $\Sigma$.
Choosing $m_n$ large enough so that the function represented by $O\left(m_n^{-\delta_n/7}\right)$ has sup-norm at most $\frac{1}{2}(1-\gamma_n)I_{\Sigma_n}$.
We see that since $U_{\nu_n}(x)
\ge I_{\Sigma_n}>0$,
$$U_{\mu_n}(x) \ge \gamma_n U_{\mu_{Q_{m_n,n}}}(x) + \frac{1}{2}(1-\gamma_n)I_{\Sigma_n}>\gamma_n U_{\mu_{Q_{m_n,n}}}(x)$$
as desired.
\newline

Now we show that $\int\log|Q_{m_n,n}(x)|d\mu_n\ge 0$. By \eqref{approx_pot}, $\int\log|Q_{m_n,n}(x)|d\mu_{m_n,n}$ is
$$\int\log|Q_{m_n,n}(x)|d\tilde\mu_n+O(m_n^{1-\delta_n/7})$$
which by \eqref{poly_pot} is
$$\deg(Q_{m_n,n})I_{\tilde\mu_n}+\int \min\left(\log|x-\tilde x|+\sum_{i>m_n}\log|x-r_i|,0\right)d\tilde\mu_n(x) + O\left(m_n^{1-\delta_n/7}\right).$$
As long as we show the integral is $o(m_n)$, then (potentially taking $m_n$ to be even larger) this value is positive and $\int\log|Q_{m_n,n}(x)|d\mu_n>0$ as well since $\nu_n$ is arithmetic.
\newline

Fix $i>m_n$. We show $\int \log|x-r_i|d\tilde\mu_n\ll_n 1$ so that $\int\sum_{i>m_n}\log|x-r_i|d\tilde\mu_n\ll_nm_n^{1-\epsilon}$ which is $o(m_n)$ as needed.
Indeed, $\left|\int \log|x-r_i|d\tilde\mu_n\right|$ is at most
$$\sum_{i=0}^\infty \left(\tilde\mu_n([r_i-c2^{1-i},r_i-c2^{-i}]) + \tilde\mu_n([r_i+c2^{-i},r_i+c2^{1-i}])\right)\left|\log|c2^{-i}|\right|\ll_n \sum_{i=0}^\infty 2^{-i\delta_n}\left|\log(c2^{-i})\right|\ll_n1$$
where $c=\max\Sigma_n$.
In the complex case, a similar computation holds by summing over annuli of radii exponentially approaching 0.
Now we look at $\int\log|x-\tilde x|d\tilde\mu_n$. We see that
$$\left|\int\log|x-\tilde x|d\tilde\mu_n \right| \le \int\big|\log|x-\overline x|\big|d\tilde\mu_n + \sum_{i>m_n}\int\big|\log|x-r_i|\big|d\tilde\mu_n\le \int\big|\log|x-\overline x|\big|d\tilde\mu_n + O_n\left(m_n^{1-\epsilon}\right)$$
where $\overline x$ is the closest root of $Q_{m_n,n}$ to $x$ excluding $r_i$ for $r_i>m_n$. This ensures that the remaining roots are separated by $\Omega\left(m_n^{-4/3}\right)$ in the real case and $\Omega\left(m_n^{-5/6}\right)$ in the complex case. Looking at the real case,
by integrating over neighborhoods of $r_i$ of order $m_n^{-4/3}$, we get
$$\int\log|x-\overline{x}|d\tilde\mu_n \le 
\sum_{i=0}^{m_n}\left|\int_{r_i-m_n^{-4/3}}^{r_i+m_n^{-4/3}}\log|x-r_i|d\tilde\mu_n\right|+O_n(\log m_n).$$
This works because any overlap in the intervals would only increase the summation, and the $O_n(\log(m_n))$ error term is achieved by integrating over any space not covered by these intervals.
Now for fixed $1\le i\le m_n$, $\displaystyle\left|\int_{r_i-m_n^{-4/3}}^{r_i+m_n^{-4/3}}\log|x-r_i|d\tilde\mu_n\right|$ is at most
$$\sum_{k=1}^\infty\left(\tilde\mu_n\left(r_i-m_n^\frac{-4k}{3},r_i-m_n^\frac{-4(k+1)}{3}\right) + \tilde\mu_n\left(r_i + m_n^\frac{-4(k+1)}{3}, r_i + m_n^\frac{-4k}{3}\right)\right) \left|\log(m_n^{-4(k+1)/3})\right|.$$
This is at most $2\sum_{k=1}^\infty m_n^\frac{-4k\delta_n}{3}\left|\log(m_n^{-4(i+1)/3})\right|\ll_nm_n^{-4\delta_n/3}\log(m_n)$. Therefore, we have that $\int\log|x-\overline{x}|d\tilde\mu_n(x)=o(m_n)$ as desired. Again a similar computation gives the same result for the complex case using summing over the integrals of annuli of decreasing radius around each $r_i$.
This completes the proof that $\int\log|Q_{m_n,n}(x)|d\mu_n\ge 0$.
\newline

It remains to study the convergence.
Note that so long as the limits exist, $\lim_{n\to\infty} \mu_n = \lim_{n\to\infty} \mu_{m_n,n}$. We also know that $\lim_{n\to\infty }\tilde\mu_n=\mu.$ So it suffices to show that for any continuous function $f$, $\lim_{n\to\infty}\int fd\mu_{m_n,n}-\int fd\tilde\mu_n=0$. We provide the argument for the complex case. The real case is almost identical except we sum over $B_i$ rather than $B_{ij}$.
\newline

Fix a continuous function $f$. Since $\bigcup_{n\in\mathbb N}(\text{supp}(\tilde \mu_n)\cup\text{supp}(\mu_{m_n,n}))$ is bounded, we have that $f$ is uniformly continuous on this set. Fix $\epsilon>0$. Suppose that for any $x$ and $y$ in the support of any of these measures then $|x-y|<\epsilon'$ implies $|f(x)-f(y)|<\frac{\epsilon}{2}$. Also let the supremum of $|f|$ on these supports be denoted by $M$. Fix $N$ such that for all $n\ge N$, we have that $\text{diam}(B_{ij})<\epsilon'$ for all $i$ and $j$ and also $m_n^{-\min(\frac{1}{2},\delta)} < \frac{\epsilon}{2M+\epsilon}$. Then
\begin{align*}
    \left|\int fd\tilde\mu_n -\int fd\mu_{m_n,n}\right| &\le \sum_{ij}\left|\int_{B_{ij}} f(x)d\tilde\mu_n(x) -\sum_{z\in B_{ij}\cap \mathcal Z_{m_n,n}}m_n\int_{\frac{-1}{2m_n^2}}^\frac{1}{2m_n^2} f(z+x)dx\right|\\
    &\le\sum_{ij}\left(\tilde\mu_n(B_{ij})\frac{\epsilon}{2} + m_n^{-\frac{2}{3}-\delta_n}\left(M+\frac{\epsilon}{2}\right)\right)\\
    &\le \frac{\epsilon}{2} + m_n^{-\min(\frac{1}{2},\delta)}\left(M+\frac{\epsilon}{2}\right)\\
    &< \epsilon
\end{align*}
This completes the proof of the theorem.
\end{proof}

\section{Formulating a linear programming problem and its dual}\label{lin_prog}
Recall the definition of the primal measure and the notations from subsection~\ref{primalintro}. We begin by showing the support of $\mu_A$ is inside a fixed compact set independent of $A$. Our proof is based on the principle of domination for potentials that we cite below from~\cite[Theorem 3.2]{logpotentials}. 
\newline

\begin{theorem}~\cite[Theorem 3.2]{logpotentials}\label{dom}
    Let $\mu$ and $\nu$ be two positive finite Borel measures with compact support on $\mathbb{C}$ and suppose that the total mass of $\nu$ does not exceed that of $\mu$. Assume further that $\mu$ has finite logarithmic energy. If the inequality
    \[
    U_{\mu}(z)\geq U_{\nu}(z)+c 
    \]
 holds $\mu$-almost everywhere for some constant $c$, then it holds for all $z\in \mathbb{C}$.
\end{theorem}

\begin{proposition}\label{compact} Fix $\Omega$ any  subset of $\mathbb{C}$. 
Let $A$ and  $\mathcal{C}(A)$ be as above and suppose that some $(\mu,b_Q)\in \mathcal{C}(A)$ satisfies 
\[
\int \log|Q(x)|d\mu(x)\geq 0
\]
for every $Q\in A$ and
\[
U_{\mu}(x)\geq \sum_{Q\in A} b_Q \log |Q(x)|
\]
for every $x\in \Omega$. 
Then there exists  $(\mu',b_Q')\in \mathcal{C}(A)$ such that $\mu'$ is supported inside the interval $[0,18]$,  
\[
\int \log|Q(x)|d\mu'(x)\geq 0
\]
for every $Q\in A$,  
\[
U_{\mu'}(x)\geq \sum_{Q\in A} b_Q' |\log Q(x)|
\]
for every $x\in \Omega$, and
\[
\int xd\mu(x)\geq \int xd\mu'(x).
\]
The above is a strict inequality if $\mu$ is not supported inside $[0,18]$. Moreover, if $\Omega$ is dense in $[0,18]$ then $(\mu',b_Q')\in \mathcal{C}(A)^+$ is a feasible point. 
\end{proposition}
\begin{proof}
If $\int x d\mu(x)\geq 2$, let $\mu'=\mu_{[0,4]}$ be the equilibrium measure of $[0,4]$ and $(b_Q')=(0,\dots,0)$.
This is sufficient to prove the lemma in this case since the potential of $\mu'$ is non-negative on $\mathbb C$.
Furthermore, if $\mu(\mathbb C)>1$, then $(\mu/\mu(\mathbb C), b_Q/\mu(\mathbb C))\in C(A)$ satisfies the requirements of the lemma and with smaller expectation. So without loss of generality, $\mu$ is a probability measure.
The rest of the proof is almost identical to that of Lemma 2.6 of \cite{lowerbound} while applying Theorem \ref{dom} to the inequality (12) in that proof.
\end{proof}

\begin{lemma}\label{lemexis}
   Recall the definition of $\Lambda_A$ from \eqref{defA}. There exists a feasible point $(\mu_A,b_Q)\in \mathcal{C}(A)^+$, where $\mu_A$ is supported in $[0,18]$  and $ \int xd\mu_A(x) = \Lambda_A.$
\end{lemma}
\begin{proof}
    By definition of $\Lambda_A$ and Proposition~\ref{compact}, there exists a sequence of probability measures $\{(\mu_n,b_{Q,n})\}$ that are supported on the interval $[0,18]$ satisfying 
    the primal constraints~\eqref{primal}, and
    \[
\lim_{n\to \infty} \int xd\mu_n(x)=\Lambda_A.
\]
By the asymptotic of the potential function $U_{\mu_n}(x)=\log|x|+O(1/x)$ as $x\to \infty$,  it follows that 
    \[
    \sum_{Q\in A} b_{Q,n}\deg(Q)\leq 1.
    \]
By compactness of $[0,1]^{|A|}$,  we may assume without loss of generality that  \[
\lim_{n\to \infty} b_{Q,n}=b_Q
\]
for some $b_Q\geq 0$ for every $Q\in A$.
Moreover, by compactness of the space of probability measures supported on $[0,18]$, there exists a sub-sequence  $\mu_{a_n}$ that converges weakly to some probability distribution $\mu$.  Note that $\log |Q(x)|$ is an upper semi-continuous function. Hence, by monotone convergence theorem, we have 
\[
\begin{split}
      U_{\mu}(x)\geq \liminf_{n\to \infty} U_{\mu_{a_n}}(x) \geq  \sum_{Q\in A}b_Q\frac{\log|Q(x)|}{\deg Q}, 
      \\
  \int \log|Q(x)|d\mu(x)\geq  \liminf_{n\to \infty}\int\log|Q(x)|d\mu_{a_n}(x) \geq 0,
\end{split}
\]
Note that $(\mu,b_Q)$ is a feasible point and $\int x d\mu=\Lambda_A.$ 
This completes the proof of our lemma. 
\end{proof}
In the next proposition, we show the existence of an explicit optimal probability measure $\mu_{A,\Sigma_A}$ to the primal problem under some assumption on the support of $\mu_A$ and its logarithmic potential in Lemma~\ref{lemexis}.
\begin{proposition}\label{analytic}
    Suppose that $A\neq \emptyset$ and $(\mu_A,b_Q)\in \mathcal{C}(A)^+$ with $\int xd\mu_A=\Lambda_A$ from Lemma~\ref{lemexis}. Furthermore, suppose that $\mu_A$  is supported on $\Sigma_A=\bigcup_{j=1}^l I_j $, a finite union of intervals away from roots of $Q\in A$. Suppose that there exists a dense set of points $K\subset \Sigma_A$ such that for any $\varepsilon>0$ and any open intervals $O\subset \Sigma_A$ then there exists $\xi\in K \cap O$ such that
    \[
    |U_{\mu_A}(\xi)-\sum_{Q}b_Q\log|Q(\xi)||\leq \varepsilon.
    \]
    There exists a  probability measure $\mu_{A,\Sigma_A}$  supported on $\Sigma_A$ such that
    \[
    U_{\mu_{A,\Sigma_A}}(x)=\sum b_Q\log|Q|(x)
    \]
    for every $x\in \Sigma_A$, and $(\mu_{A,\Sigma_A},b_Q)\in \mathcal{C}(A)^+$ with $\int x d\mu_{A,\Sigma_A}(x)=\Lambda_A.$
\end{proposition}
\begin{proof}
We note that $\mu_A$ is an arithmetic measure supported on $\Sigma_A$ with non-constant potential.
This implies that
\[
I(\mu_A)=\int \log|x-y|d\mu_A(x)d\mu_A(y)\geq 0
\]
and 
 $C_{\Sigma_A}>0$, where $C_{\Sigma_A}$ is the logarithmic capacity of $\Sigma_A$. 
First, we construct the  measure $\mu_{A,\Sigma_A}$. By ~\cite[Lemma 3.2]{lowerbound}, there exists a probability measure $\mu_{Q,\Sigma_A}$ with potential
 \[
 U_{\mu_{Q,\Sigma_A}}(x)=\frac{\log|Q|(x)}{\deg(Q)}-C_{Q,\Sigma_A}
 \]
where $C_{Q,\Sigma_A}>0.$ Let $\mu_{eq,\Sigma_A}$ be the equilibrium probability measure on $\Sigma_A.$ Let 
\[
\mu_{A,\Sigma_A}:=\sum_{Q\in A}b_Q\deg(Q)\mu_{Q,\Sigma_A}+b_{eq}\mu_{eq,\Sigma_A}
\]
where
\[
b_{eq}:=\sum_{Q\in A}b_Q\deg(Q)\frac{C_{Q,\Sigma_A}}{C_{\Sigma_A}}>0.
\]
Since $b_Q\geq 0$ and $b_{eq}>0$, $\mu_{A,\Sigma_A}$ is a positive measure. By the above choice $b_{eq}$, we have
 \[
    U_{\mu_{A,\Sigma_A}}(x)=\sum b_Q\log|Q|(x)
    \]
    for every $x\in \Sigma_A$. Suppose that $f(x)$ is a continuous function on $\Sigma_A$ such that
    \[
    f(x)=\int_{\Sigma_A} \log|x-y|g(y)dy
    \]
     for some function $g(y)$ on $\Sigma_A$, where  $ \int \log|x-y||g(y)|dy< \infty$ for every $x\in\Sigma_A$.
Next, we show 
\[
\int_{x\in \Sigma_A} f(x)d \mu_{A,\Sigma_A}(x)=\int_{x\in \Sigma_A} f(x)d \mu_{A}(x). 
\]
We rewrite the above as 
\[
\int f(x)d \mu_{A,\Sigma_A}(x)= \int_{x\in \Sigma_A} \int_{y\in \Sigma_A} \log |x-y|  g(y)dy d\mu_{A,\Sigma_A}(x)
\]
Since
\(
\int |\log |x-y|g(y)| dy d\mu_{A,\Sigma_A}(x) < \infty
\), by Fubini's theorem,
\[
\int_{x\in \Sigma_A}f(x) d \mu_{A,\Sigma_A}(x) = \int_{y\in \Sigma_A} \sum_{Q\in A} b_Q\log|Q|(y)g(y)dy. 
\]
Given $N\geq 0$, we partition $\Sigma_A\subset \bigcup_{i=1}^M O_i$ into open intervals of length less than $1/N$. By our assumption, for every open interval $O_i$, there exists $\xi_{i,N}\in O_i \cap K $ such that
\[
    |U_{\mu_A}(\xi_{i,N})-\sum_{Q\in A}b_Q\log|Q(\xi_{i,N})||\leq 1/N.
\]
By  approximating the integral  by Riemann sum over $\xi_{i,N} \in O_i$, we have
\[
\begin{split}
    \int_{x\in \Sigma_A}f(x) d \mu_{A,\Sigma_A}(x)&=\int_{y\in \Sigma_A} \sum_{Q\in A} b_Q\log|Q|(y)g(y)dy
    \\
    &= \lim_{N\to \infty} \sum_{i}\sum_{Q}b_Q\log|Q(\xi_{i,N})| g(\xi_{i,N}) \text{len}(O_i)
    \\
    &=
    \lim_{N\to \infty} \sum_{i}U_{\mu_A}(\xi_{i,N})g(\xi_{i,N}) \text{len}(O_i)
    \\
    &= \int_{x\in \Sigma_A} \int_{y\in \Sigma_A} \log |x-y|  g(y)dy d\mu_{A}(x)
    \\
 &=\int f(x)  d\mu_A(x).
    \end{split}
\]
By taking $f(x)=1$ and $g(y)dy=\frac{1}{C_{\Sigma_A}}d\mu_{eq,\Sigma_A}(y)$, we have
\[
\int d \mu_{A,\Sigma_A}(x)=\int   d\mu_A(x)=1.
\]
Similarly, by taking $f(x)=\frac{\log|Q(x)|}{\deg(Q)}$ and $g(y)dy=d\mu_{Q,\Sigma_A}+\frac{C_{Q,\Sigma_A}}{C_{\Sigma_A}}$, we have
\[
\int \log|Q(x)| d \mu_{A,\Sigma_A}(x)=\int \log |Q(x)|  d\mu_A(x)\geq 0.
\]
By~\cite[Lemma 3.3]{lowerbound}, there exists $g(y)dy $ such that 
\[
x=\int_{\Sigma_A}\log|x-y|g(y)dy.
\]
This implies that
\[
\int xd \mu_{A,\Sigma_A}(x)=\int   xd\mu_A(x)=\Lambda_A.
\]
Finally, let $\eta:=\sum_{Q} b_Q\deg(Q)\mu_Q$. Note that $|\eta|=\sum b_Q\deg(Q) \leq 1$. 
By Theorem~\ref{dom}, since  $\mu_{A,\Sigma_A}$ is a probability measure and $U_{\mu_{A,\Sigma_A}}(x)= \sum b_Q\log|Q|(x)=U_{\eta}(x)$ for every $x\in \Sigma_A$, 
we have
\[
    U_{\mu_{A,\Sigma_A}}(x)\geq \sum b_Q\log|Q|(x)=U_{\eta}(x)
\]
for every $x\in\mathbb{C}$.
This implies that $(\mu_{A,\Sigma_A},b_Q)\in \mathcal{C}(A)^+$ and concludes our proposition. 
\end{proof}

\subsection{Dual problem}
In this subsection, we define the associated dual optimization problem. 
Let 
\[
\mathcal{D}(A):=\{(\nu,\lambda_Q,\beta_Q,\lambda_0): \nu\in\mathcal{M}^+, \beta_Q,\lambda_Q\geq 0 \text{ for every } Q\in A,\text{ and }\lambda_0\geq 0\}.
\]
For $(\mu,b_Q)\in \mathcal{C}(A)$ and $(\nu,\lambda_Q,\beta_Q,\lambda_0)\in \mathcal{D}(A)$, we define the  Lagrangian function as
\begin{multline}\label{Lag}
L\left((\mu,b_Q),(\nu,\lambda_Q,\beta_Q,\lambda_0)\right):=\int xd\mu (x)-\int\left( U_{\mu}(x)-\sum_{Q\in A} b_Q\log|Q(x)|\right)d\nu(x)
\\
-\sum_{Q\in A}\left(\lambda_Q\int \log|Q(x)|d\mu(x)+b_Q\beta_Q\right)-\lambda_0\left(\int d\mu -1\right).
\end{multline}
We note that if $(\mu,b_Q)\in \mathcal{C}(A)^+$ is a feasible point then 
\begin{equation}\label{Lagin}
  L\left((\mu,b_Q),(\nu,\lambda_Q,\beta_Q,\lambda_0)\right) \leq \int xd\mu (x)   
\end{equation}

for every $(\nu,\lambda_Q,\beta_Q,\lambda_0)\in \mathcal{D}(A).$ We may rewrite the  Lagrangian function as
\begin{multline}\label{Lag2}
L\left((\mu,b_Q),(\nu,\lambda_{Q},\beta_{Q},\lambda_{0})\right)=\lambda_{0}+\int \left(x-U_{\nu}(x)-\sum_{Q\in A}\lambda_{Q}\log|Q(x)|-\lambda_{0} \right)  d\mu (x)
\\
+\sum_{Q\in A} \left(\int \log|Q(x)|d\nu(x)-\beta_{Q}\right) b_{Q}.
\end{multline}
We define
\[
g((\nu,\lambda_Q,\beta_Q,\lambda_0)):= \inf_{(\mu,b_Q)\in \mathcal{C}(A)} L\left((\mu,b_Q),(\nu,\lambda_Q,\beta_Q,\lambda_0)\right).
\]
\begin{lemma}\label{duallem}
    Suppose that $g((\nu,\lambda_Q,\beta_Q,\lambda_0))>-\infty$. Then \begin{equation}\label{dualconst}
    \begin{split}
        \int \log|Q(x)|d\nu(x)=\beta_{Q}\geq 0,
        \\
      x-U_{\nu}(x)-\sum_{Q\in A}\lambda_{Q}\log|Q(x)| \geq \lambda_{0} \geq  0 
    \end{split}
\end{equation}
for every $x\in\mathbb{R}^+$. We denote the above by the \textit{dual constraints}.  Moreover, we have
\[
g((\nu,\lambda_Q,\beta_Q,\lambda_0))=\lambda_0.
\]
\end{lemma}
\begin{proof}
 Since the minimum of a non-zero linear function is $-\infty$, by \eqref{Lag2} $g((\nu,\lambda_Q,\beta_Q,\lambda_0))=-\infty$ unless  $\int \log|Q(x)|d\nu(x)=\beta_{Q}$. Suppose to the contrary that for some $x_0>0$,
 $$
x_0-U_{\nu}(x_0)-\sum_{Q\in A}\lambda_{Q}\log|Q(x_0)| <0.
 $$
  By taking $(\mu,b_Q)=(a\delta_{x_0},0,\dots,0)$ for arbitrary large $a>0$, it follows that 
 \[
 g((\nu,\lambda_Q,\beta_Q,\lambda_0))=-\infty.
 \]
 Finally, by the above inequalities and~\eqref{Lag2}, we have
 \[
 L\left((\mu,b_Q),(\nu,\lambda_{Q},\beta_{Q},\lambda_{0})\right)\geq \lambda_0
 \]
 for every $(\mu,b_Q)\in \mathcal{C}(A)$ with equality for $(\mu,b_Q)=(0,\dots,0)$.
\end{proof}

We say $(\nu,\lambda_Q,\beta_Q,\lambda_0)$ is a feasible dual point if $(\nu,\lambda_Q,\beta_Q,\lambda_0)$ satisfies all the dual constraints. We denote the set of dual feasible points by $\mathcal{D}(A)^+.$ 
\begin{proposition}\label{smooth}
    Suppose that $(\nu,\lambda_Q,\beta_Q,\lambda_0)\in \mathcal{D}(A)^+$ is a feasible point. We have
    \[
2e\geq  \lambda_0+\sum_{Q}\deg(Q)\lambda_Q+\nu(\mathbb{R}).
\]
Moreover, suppose that 
    \[
    d_{Q_0}:=\lambda_0- \Tr(Q_0)>0
    \]
     for some $Q_0\in A$ with all roots being  positive real. Then 
    \[
    \lambda_{Q_0}\gg d_{Q_0}/\max(\log(1/d_{Q_0}),1),
    \]
    where the implicit constant is universal and independent of all variables.
\end{proposition}
\begin{proof}
Note that 
\begin{equation}\label{ddd}
 x\geq \lambda_0+ \sum_{Q}\lambda_Q \log|Q(x)|+U_{\nu}(x)    
\end{equation}
for every $x\in \mathbb{R}^+$. We integrate the above inequality with respect to $\mu_{[0,4e]}$, the equilibrium measure of the interval $[0,4e]$, and obtain
\[
2e\geq  \lambda_0+\sum_{Q}\deg(Q)\lambda_Q+\nu(\mathbb{R}).
\]
    Let $w_{\varepsilon}$ be the compactly supported probability measure supported on $[\varepsilon^2,\varepsilon]$ with logarithmic potential
    
    \begin{equation}\label{weight}
    U_{w_{\varepsilon}}(x)=\log|x|-C_{\varepsilon}
    \end{equation}

    for every $x\in [\varepsilon^2,\varepsilon]$
    where 
 \begin{equation}\label{eps}
     \varepsilon^{1/2}\gg C_{\varepsilon}=\log\left(\frac{\varepsilon^2+2\varepsilon^{3/2}+\varepsilon}{\varepsilon-\varepsilon^2}\right)>0.
 \end{equation}
    Let $\mu_{Q_0}=\frac{1}{\deg(Q_0)}\sum_{Q_0(\alpha)=0}\delta_{\alpha}$ be the root distribution of $Q_0$. We note that 
    \[
    \lim_{\varepsilon\to 0} w_{\varepsilon}*\mu_{Q_0}=\mu_{Q_0}
    \]
    where $dw_{\varepsilon}*\mu_{Q_0}(x)=\frac{1}{\deg(Q_0)}\sum_{Q_0(\alpha)=0} dw_{\varepsilon}(x-\alpha).$ Moreover, for every $x\in \mathbb{C}$, we have
    \[
    U_{w_{\varepsilon}*\mu_{Q_0}}(x) \geq \max\left(\frac{\log |Q_0(x)|}{\deg(Q_0)}-C_{\varepsilon}, 2\log(\varepsilon)-C_{\varepsilon}\right).
    \]

    This implies that 
    \[
    \int U_{\nu}(x) dw_{\varepsilon}*\mu_{Q_0}(x)=\int U_{w_{\varepsilon}*\mu_{Q_0}}(x) d\nu(x)\geq \int \frac{\log |Q_0(x)|}{\deg(Q_0)}-C_{\varepsilon} d\nu(x) \geq -C_{\varepsilon}.
    \]
    Moreover, 
    \[
    \int \log|Q(x)|d\mu_{Q
    _0}(x) \geq 0
    \]
    for every non-zero $Q\neq Q_0$, and hence
    \[
    \int \log|Q(x)|dw_{\varepsilon}*\mu_{Q_0}(x) \geq -C_{\varepsilon}
    \]
    for every $Q\in A\setminus\{Q_0\}$ and 
    \[
    \int xdw_{\varepsilon}*\mu_{Q_0}(x) \leq \varepsilon+\Tr(Q).
    \]
    On the other hand, average the  inequality~\eqref{ddd} with respect to $dw_{\varepsilon}*\mu_{Q_0}(x)$, and obtain
    \[
    \varepsilon+\Tr(Q) \geq  \lambda_0- C_{\varepsilon}\sum_{Q\neq Q_0} \lambda_Q-\lambda_{Q_0}(2\log(1/\varepsilon)-C_{\varepsilon})
    \]
    By taking $\varepsilon_0\gg (\lambda_0-\Tr(Q))^2 $ small enough, where $C_{\varepsilon_0}<\frac{\lambda_0-\Tr(Q)}{2(1+\sum_{Q\neq Q_0} \lambda_Q)}$, the above inequality implies that
    \[
    \lambda_{Q_0}\geq \frac{\lambda_0-\Tr(Q_0)}{4\log(1/\varepsilon_0)}\gg d_{Q_0}/\log(1/d_{Q_0}).
    \]
    This completes the proof of our proposition.
\end{proof}
\begin{proposition}\label{conv}
      Suppose that $(\nu,\lambda_Q,\beta_Q,\lambda_0)\in \mathcal{D}(A)^+$ is a feasible point, and $x\in A$. Let $w_{\varepsilon}$ be the smooth probability measure supported on $[\varepsilon^2,\varepsilon]$ defined in \eqref{weight}, where $U_{w_{\varepsilon}}(x)=\log|x|-C_{\varepsilon}$ for every $x\in [\varepsilon^2,\varepsilon]$. 
    Define
      $\nu_{\varepsilon}:=\nu*w_{\varepsilon}+C_{\varepsilon}|\nu|\mu_{[0,18]}$, where $|\nu|:=\int d\nu$ and  $\mu_{[0,18]}$ is the equilibrium measure of the interval $[0,18]$. We have
      \[
\lim_{\varepsilon\to 0}U_{\nu_{\varepsilon}}(x)=U_{\nu}(x),
      \]
      for every $x\in \mathbb{C},$
      \[
(\nu_{\varepsilon},\lambda_Q,\beta_{Q,\varepsilon},\lambda_{0,\varepsilon})\in \mathcal{D}(A)^+,
      \]
      where 
      \(
        \beta_{Q,\varepsilon}:=\int \log|Q(x)|d\nu_{\varepsilon}(x)
      \)
      and 
      \[
      \lambda_{0,\varepsilon}\geq \lambda_0-C \varepsilon^{1/2}.
      \]
      for some absolute constant $C$.
\end{proposition}
\begin{proof}Note that $\lim_{\varepsilon\to 0}\nu_{\varepsilon}=\nu$. This implies that for every $x\in \mathbb{C}$, we have
\[
\limsup_{\varepsilon\to 0} U_{\nu_{\varepsilon}}(x)\leq U_{\nu}(x).
\]
On the other hand, we have
\[
U_{\nu_{\varepsilon}}(x)=\int U_{w_{\varepsilon}}(x-y)d\nu(y)+C_{\varepsilon}|\nu|U_{[0,18]}(x).
\]
By~\eqref{weight} and \eqref{eps}, we have
\[
U_{\nu_{\varepsilon}}(x)\geq U_{\nu}(x)-O(\varepsilon^{1/2}). 
\]
This implies that
\[
\liminf _{\varepsilon\to 0}U_{\nu_{\varepsilon}}(x)\geq U_{\nu}(x).
\]
Hence,
      \[
\lim_{\varepsilon\to 0}U_{\nu_{\varepsilon}}(x)=U_{\nu}(x).
      \]
Note that $\nu*w_{\varepsilon}$ is supported on $\mathbb{R}^+$ and 
\[
U_{\nu*w_{\varepsilon}}(x)= \int U_{w_\varepsilon}(x-y_1)d\nu(y_1)\geq \int \log|x-y_1|-C_{\varepsilon}d\nu(y_1)
\geq U_{\nu}(x)-C_{\varepsilon}|\nu|
\]

We show that $\beta_{Q,\varepsilon}\geq 0$. We have
\[
\begin{split}
      \beta_{Q,\varepsilon}=\int \log|Q(x)|d\nu_{\varepsilon}(x)&=\frac{1}{\deg(Q)}\sum_{\alpha,Q(\alpha)=0}U_{\nu_{\varepsilon}}(\alpha)
      \\
      &\geq \frac{1}{\deg(Q)}\sum_{\alpha,Q(\alpha)=0}U_{\nu}(\alpha)-C_{\varepsilon}|\nu|+C_{\varepsilon}|\nu| \log(9/2) \geq 0
      \end{split}
\]
where we used $\int \log|Q(x)|d \nu(x) \geq 0$. By Lemma~\ref{duallem}, it is enough to show that
    \[
    x-\sum \lambda_Q \log|Q(x)|-\lambda_0- U_{\nu_\varepsilon(x)}\geq -C\varepsilon^{1/2}.
    \]
    for some constant $C>0$ and every $x\in \mathbb{R}^+.$
    By Proposition~\ref{smooth}, we have $\lambda_{x}\neq 0$ and the above inequality holds for $0<x\ll 1$. Similarly, the above inequality holds when $x$ is near a root of $Q$. So, without loss of generality we   suppose that $x\gg 1$ and $|x-\alpha|\gg 1$ for any $\alpha$ where $Q(\alpha)=0$ for some $Q\in A.$
    By our assumption, we have 
    \[
    (x-y)-\sum \lambda_Q \log|Q(x-y)|-\lambda_0- U_{\nu}(x-y)\geq 0
    \]
    for every $y<x$.  We integrate the above with respect to $dw_\varepsilon(y)$ and obtain
    \[
     x-\sum \lambda_Q \log|Q(x)|-\lambda_0- U_{\nu*w_\varepsilon(x)}   \gg -\varepsilon
    \]
    The above implies that
    \[
    x-\sum \lambda_Q \log|Q(x)|-\lambda_0- U_{\nu_\varepsilon(x)}   \gg -\varepsilon^{1/2}.
    \]
    This completes the proof of our proposition.
\end{proof}

The dual problem is to maximize $g((\nu,\lambda_Q,\beta_Q,\lambda_0))$ for $(\nu,\lambda_Q,\beta_Q,\lambda_0)\in \mathcal{D}(A)^+$. Let
\begin{equation}\label{dualdef}
\lambda_A:=\sup_{(\nu,\lambda_Q,\beta_Q,\lambda_0)\in \mathcal{D}(A)} g((\nu,\lambda_Q,\beta_Q,\lambda_0)).  
\end{equation}
Note that $g(0,\dots,0)=0$, which implies $\lambda_A\geq0$. In Proposition~\ref{mainprop}, we show that there exists a   $(\nu_A,\lambda_{A,Q},\beta_{A,Q},\lambda_{A})\in  \mathcal{D}(A)^+$ with 
\[
g(\nu_A,\lambda_{A,Q},\beta_{A,Q},\lambda_{A})=\lambda_{A}.
\]

\subsection{Strong duality}
By \eqref{Lagin}, for every $(\mu,b_Q)\in \mathcal{C}(A)^{+}$ and every $(\nu,\lambda_Q,\beta_Q,\lambda_0)\in \mathcal{D}(A)$, we have
\[
g((\nu,\lambda_Q,\beta_Q,\lambda_0)) \leq \int xd\mu (x).
\] 
This implies
\(
\lambda_A  \leq \Lambda_A,
\)
 which is known as the weak duality inequality. 
In Proposition~\ref{mainprop}, we show that strong duality holds which means $\lambda_A = \Lambda_A.$
\begin{proposition}\label{mainprop}
    Fix a finite subset of integral polynomials $A$ with all real roots and let $\Lambda_A$ and $\lambda_A$ be defined as in~\eqref{defA} and~\eqref{dualdef}. Then $\Lambda_A=\lambda_A$. Moreover, there is a  solution $(\mu_A,b_Q)\in \mathcal{C}(A)^+$ to the primal problem associated to $A$ with $\int x d\mu_A(x)=\Lambda_{A}$. Furthermore, $\mu_A$ is supported on a finite union of intervals $\Sigma\subset [0,18]$   and every gap between the intervals contains a root of some $Q\in A$, and
    \[
    \begin{array}{cc}
          \int x d\mu_A(x)=\Lambda_{A},\\
         I(\mu_A,\nu_A)=0,          \\
    U_{\mu_A}(x)\geq \sum_{Q\in A} b_Q \log|Q(x)|,\\
          \int \log|Q(x)| d \mu_A=0 \text{ if $\lambda_Q\neq 0$,}

    \end{array}
    \]
where equality holds in the third line for every $x\in \Sigma_A$ with some scalars $b_Q\geq 0$. Furthermore, there exists   $(\nu_A,\lambda_{A,Q},\beta_{A,Q},\Lambda_{A})\in  \mathcal{D}(A)^+$  such that $\nu_A$ is supported on $\Sigma_A$,
       \[
    \begin{array}{cc}
                   x\geq \Lambda_A +\sum_{Q\in A}\lambda_{A,Q} \log|Q(x)|+ U_{\nu_A}(x),\text{ and}\\
          \beta_{A,Q}=\int \log|Q(x)| d \nu_A=0 \text{ if $b_Q\neq 0$,}

    \end{array}
    \]
where in the first line the inequality holds for non-negative $x$ and equality holds for every $x\in \Sigma_A$ with some scalars $\lambda_{A,Q}\geq 0$.  
\end{proposition}

\begin{proof}[Proof of Proposition~\ref{mainprop}]  Let $A=\{Q_1,\dots,Q_n \}$. 
We note that $\Lambda_A\leq 2$, since $(\mu_{[0,4]},0,\dots,0)$, where $\mu_{[0,4]}$ is the equilibrium measure of the interval $[0,4]$ is a feasible point and has expected value $2$. 
We fix an enumeration of positive rational numbers which are not integers and write 
\[
\mathbb{Q}^+\setminus \mathbb Z:=\{x_i: i\geq 0\}.
\]
For $N\geq 1$, we define $\psi_N:\mathcal{C}(A)\to \mathbb{R}^{2+2n+N}$ as follows
\begin{align*}
\psi_N((\mu,b_Q)):=(E_{\mu}(x),E_{\mu}(1),E_{\mu}(\log|Q_1|),\dots,E_{\mu}(\log|Q_n|),b_{Q_1},\dots,b_{Q_n},
\\
U_{\mu}(x_0)-\sum_j b_{Q_j}\log |Q_j(x_0)|,\dots, U_{\mu}(x_N)-\sum_j b_{Q_j}\log |Q_j(x_N)|)    
\end{align*}
where $E_\mu(f):=\int f(x)d\mu(x)$ is the expected value of  $f$ with respect to  $\mu$. Let 
\[\mathcal{A}_N:=\psi_N(\mathcal{C}(A)) \subset  \mathbb{R}^{2+2n+N}\]
be the image of $\mathcal{C}(A)$ by $\psi_N$. For $\varepsilon>0$, let 
\[
\mathcal{B}_{\varepsilon,N}:=\{(y_0,y_1,\dots,y_{2n+N+1}): y_0\leq \Lambda_A-\varepsilon, y_1\geq 1, \text{ and }  y_i\geq 0 \text{ for }2\leq  i\leq 2n+N+1 \}.
\]
Since $\mathbb{Q}$ is dense inside $\mathbb{R}$, it follows from the definition of $\Lambda_A$ that for every $\varepsilon>0$, there exists a large enough $M$ such that 
\[
\mathcal{A}_N \cap \mathcal{B}_{\varepsilon,N} =\emptyset
\]
for every $N\geq M.$
Otherwise, there exists a sequence $(\mu_N,b_{N,Q})\in \mathcal{C}(A)$ for $N\geq 0$ such that
\[
\begin{split}
    \int x d\mu_N \leq \Lambda_A-\varepsilon
    \\
        \int  d\mu_N \geq 1
        \\
        \int \log|Q(x)|d\mu(x)\geq 0,
    \text { and }
    b_{N,Q} \geq 0 \text{ for every } Q\in A
    \\
    U_{\mu_N}(x_i)\geq \sum_j b_{N,Q_j}\log |Q_j(x_i)| \text{ for every } 0\leq i\leq N.
\end{split}
\]
By normalizing $\mu_N/\int d\mu_N$, we may assume that $\mu_N$ is a probability measure for every $N\geq 0$. By Proposition~\ref{compact}, we may assume that $\mu_N$ is supported on $[0,18]$ satisfying all the above inequalities. Note that as $x_i\to \infty$
\[
\log(x_i)+o(1)\geq  U_{\mu_N}(x_i)\geq \sum_Q b_{N,Q}\log |Q_j(x_0)|\geq \log|x_i|\sum_j b_{N,Q}\deg(Q)+o(1)
\]
This implies that
\[
0\leq \sum_Q b_{N,Q}\deg(Q) \leq 1.
\]
Moreover, since $E_{\mu_N}(x) <2$,  $\{\mu_N\}$ is a tight family of  probability measures on $\mathbb{R}^+$.  Prokhorov's theorem implies $(\mu_N,b_{N,Q})\in \mathcal{C}(A)$ has a convergent sub-sequence $(\mu_{a_n},b_{a_n,Q})$ such that
\[
\lim_{n\to \infty} \mu_{a_n} =\mu
\]
and
\[
\lim_{n\to \infty}b_{a_n,Q_j}=b_{Q_j}
\]
for some probability measure $\mu$ supported inside $[0,18]$ and $b_Q\geq 0$. Moreover, since $\log |Q(x)|$ is an upper semi-continuous function and $\mathbb{Q}\cap [0,18]$ is dense in $[0,18]$, the monotone convergence theorem implies that
\[
U_{\mu}(x)\geq \sum_j b_{Q_j}\log |Q_j(x)|
\]
for every $x$ inside the support of $\mu$. By Theorem~\ref{dom}, the above holds for every $x\in \mathbb{C}$. Note that  $(\mu,b_Q)$  satisfies all the primal problem constraints  with expected value less than $\Lambda_A-\varepsilon$ which is a contradiction by the definition of $\Lambda_A$.  
\\

Suppose that $N\geq M$. Note that $\mathcal{A}_N, \mathcal{B}_{\varepsilon,N}\subset \mathbb{R}^{2+2n+N} $ are both convex regions with \[
\mathcal{A}_N \cap \mathcal{B}_{\varepsilon,N} =\emptyset.
\]
By the separating hyperplane theorem, it follows that there exists some coefficients 
\[
(c_0,c_1,c_{Q_1}\dots,c_{Q_n},l_{Q_1},\dots,l_{Q_n},d_1,\dots, d_N)\]
depending on $N$ and $\varepsilon$ such that for every $(y_0,y_1,\dots,y_{1+2n+N})\in \mathcal{B}_{\varepsilon,N}$ and $(\mu,b_Q)\in \mathcal{C}(A)$,
\begin{multline}
\label{hyp}
  c_0y_0+c_1y_1+\sum_{i=1}^n c_{Q_i} y_{i+1} +\sum_{i=1}^n l_i y_{n+i+1}+\sum_{j=1}^N d_j y_{j+2n+1} \geq c_0E_{\mu}(x)+c_1E_{\mu}(1)+\sum_{i=1}^n c_{Q_i}E_{\mu}(\log|Q_i|)
  \\
 + \sum_{i=1}^{n}l_{Q_i}b_{Q_i} +\sum_{i=1}^N d_i\left(U_{\mu}(x_i)-\sum_j b_{Q_j}\log |Q_j(x_i)|\right).     
\end{multline}
 Since $y_i\geq 0$ could be arbitrary large for $i\geq 1$, so we must have $c_1,c_{Q_i},l_i,d_j\geq 0$ for every $i,j\geq 1$.
 Similarly, $y_0\leq\Lambda_A -\varepsilon $ could get close to $-\infty$, we must have $c_0\leq 0$. Next, we show that $c_0\neq 0.$ Let $\mu_{[0,18]}$ be the
the equilibrium measure of $[0,18]$. We note that $(\mu_{[0,18]},0,\dots,0)\in \mathcal{C}(A) $, and 
\[
E_{\mu_{[0,18]}}(x)=9,
\]
\[
E_{\mu_{[0,18]}}(\log|Q_i|)\geq \deg(Q_i)\log(9/2),
\]
\[
U_{\mu}(x_i)\geq \log(9/2).
\]
By substituting $(y_0,y_1,\dots,y_{2n+N})=(\Lambda_A-\varepsilon,1,0,\dots,0)$  and $(\mu_{[0,18]},0,\dots,0)$ in~\eqref{hyp} and the above identities, we have
\begin{equation}\label{ub}
    c_0(\Lambda_A-\varepsilon-9) \geq \sum_{i=2}^{n+1}c_{Q_i}\deg(Q_i)\log(9/2)+\sum_{j=1}^N d_j\log(9/2)/2.
\end{equation}
Since $\Lambda_A\leq 2$, the above inequality implies that $c_0\neq 0$ and $c_0<0$. By substituting $(y_0,y_1,\dots,y_{2n+N})=(\Lambda_A-\varepsilon,1,0,\dots,0)$ and dividing both sides of \eqref{hyp} by $-c_0$, we have
\begin{equation}\label{lagrang}
E_{\mu}(x)  \geq \Lambda_A-\varepsilon+c_1'( E_{\mu}(1)-1)+\sum_{i=1}^{n} c_{Q_i}'E_{\mu}(\log|Q_i|)+ \sum_{i=1}^{n}l'_{Q_i}b_{Q_i}+\sum_{i=1}^N d_i'\left(U_{\mu}(x_i)-\sum_j b_{Q_j}\log |Q_j(x_i)|\right) 
\end{equation}
for every $(\mu,b_Q)\in \mathcal{C}(A)$, where $c_{Q_i}':=-\frac{c_{Q_i}}{c_0}\geq 0$, $l_{Q_i}':=-\frac{l_{Q_i}}{c_0}\geq 0$ and $d_j':=-\frac{d_j}{c_0}\geq 0$. Let 
\[
\nu_N:=\sum_{j}d_j'\delta_{x_i}.
\]
By rewriting~\eqref{lagrang}, in terms of the Lagrangian~\eqref{Lag}, we obtain
\begin{equation}\label{mainineq}
g(\nu_N,c_{Q_i}',l_{Q_i}',c_1') \geq \Lambda_A-\varepsilon.
\end{equation}

By Lemma~\ref{duallem}, we have
\[
\int \log|Q_i(x)|d\nu_N(x)=l_{Q_i}'.
\]
\[
\Lambda_A \geq c_1'=g(\nu_N,c_{Q_i}',l_{Q_i}',c_1')\geq \Lambda_A-\varepsilon.
\]
\[
x-U_{\nu_N}(x)-\sum_{p\in A}c'_{Q_i}\log|Q_i(x)|\geq c_1'
\]
for every $x\in \mathbb{R}^+$. Moreover, by \eqref{ub}, the total mass of $\nu_N$ and sum of $c_Q$ is bounded explicitly by
\begin{equation}\label{xub}
\frac{9}{\log(9/2)/2} \geq \sum_{i=1}^{n}c'_{Q_i}\deg(Q_i)+\nu_{N}(\mathbb{R}^+).    
\end{equation}
 Next, we take a sequence of $\varepsilon_i \to 0$, $N_i\to \infty$, and construct $\nu_{N_i}$ as above such that 
\[
g(\nu_{N_i},c_{Q,i}',l_{Q,i}',c_{1,i}')=c_{1,i}'\geq \Lambda_A-\varepsilon_i.
\]

Let $\mu_A$ be the probability measure from Lemma~\ref{lemexis} with support $\Sigma_A\subset[0,18]$, it follows from~\eqref{Lagin} and~\eqref{mainineq} that 
\begin{equation}\label{sand}
    \Lambda_A \geq L\left((\mu_A,b_{Q,A}),(\nu_{N_i},c_{Q,i}',l_{Q,i}',c_{1,i}')\right)\geq \Lambda_A -\varepsilon_i
\end{equation}
This implies that
\begin{equation}\label{abinq}
    \varepsilon_i\geq -\int\left( U_{\mu_A}(x)-\sum_{Q\in A} b_Q\log|Q(x)|\right)d\nu_{N_i}(x)
\\
-\sum_{Q\in A}\left(c_{Q,i}'\int \log|Q_i(x)|d\mu_A(x)+b_Ql_{Q,i}'\right)+\varepsilon_i\geq 0.
\end{equation}
We note that $U_{\mu_A}(x)-\sum_{Q\in A} b_Q\log|Q(x)|$ is a strictly positive harmonic function outside the support of $\mu_A$. As $\varepsilon_i\to 0$, the  inequality \eqref{abinq} implies that 
\begin{equation}\label{tight}
\lim_{N_i\to \infty}\nu_{N_i}(\mathbb{R}^+\setminus \Sigma_A)=0.  \end{equation}

Otherwise, there exists an open set $U$ such that $\limsup_{N\to \infty}\nu_{N_i}(U)\gg 1$ and
\[
U_{\mu_A}(x)-\sum_{Q\in A} b_Q\log|Q(x)\gg 1
\]
for every $x\in U$. This is a contradiction with \eqref{abinq}. The inequality~\eqref{tight} implies that $\nu_{N_i}$ is a tight family of probability measures. By Prokhorov's theorem, there exists a convergent sub-sequence such that 
\[
\lim_{i\to \infty} (\nu_{N_i},c_{Q,i}',l_{Q,i}',c_{1,i}')= (\nu_A,\lambda_Q,\beta_Q,\Lambda_A),
\]
where $\nu_A$ is supported inside $\Sigma_A'\subset\Sigma_A$. Next, we show that $$g(\nu_A,\lambda_Q,\beta_Q,\Lambda_A)=\Lambda_A.$$ Fix $\varepsilon>0$, and  let $w_{\varepsilon}$ be a smooth weight function that we introduce in Proposition~\ref{conv}, and  recall the following definition from Proposition~\ref{conv} for a measure $\nu$
\[
\nu_{\varepsilon}:=\nu*w_{\varepsilon}+C_{\varepsilon}|\nu|\mu_{[0,18]}.\]
By the weak convergence of $\{\nu_{N_i}\}_{i\in\mathbb N}$ to $\nu_{A}$, we have
\[
U_{\nu_{A,\varepsilon}}(x)=\lim_{N_i\to \infty}U_{\nu_{N_i,\varepsilon}}(x)
\]
for every $x\in \mathbb{C}$ and fixed $\varepsilon>0$. By Proposition~\ref{conv}, we have 
\[
g(\nu_{N_i,\varepsilon},c_{Q,i}',l_{Q,i,\varepsilon}',c_{1,i,\varepsilon}')\geq \Lambda_A-\varepsilon_i-C\varepsilon^{1/2}
\]
for some constant $C$. This implies that 
\[
g(\nu_{A,\varepsilon},\lambda_Q,\beta_{Q,\varepsilon},\Lambda_{A,\varepsilon})\geq \Lambda_A-C\varepsilon^{1/2}.
\]

   By  Proposition~\ref{conv}, $\lim_{\varepsilon\to 0}U_{\nu_{A,\varepsilon}}(x)=U_{\nu_{A}}(x)$. Hence, the above implies that

   \begin{equation} \label{StrD}g(\nu_A,\lambda_Q,\beta_Q,\Lambda_A)=\Lambda_A.  
   \end{equation}

Next, we show that the support of $\nu_A$ is equal to the support of $\mu_A$. It is enough to show that $\Sigma_A\subset\Sigma_A'$.   By the above, \eqref{Lag2} and \eqref{sand}, we have
\begin{equation}\label{zeroid}
 \int \left|x-U_{\nu_{A}}(x)-\sum_{Q\in A}\lambda_Q\log|Q(x)|-\Lambda_A\right|  d\mu_A (x)=0.    
\end{equation}
Let
\begin{equation}\label{F}
F(z):=z-U_{\nu_A}(z)-\sum_{Q\in A}\lambda_Q\log|Q(z)|-\Lambda_A.
\end{equation}
Note that $F(z)=0$ for every $x\in \Sigma_A$ and $F(z)\geq 0$ for every $z\in \mathbb{R}^+$. 
Suppose to the contrary that $\Sigma_A'\neq \Sigma_A$. Then there exists  an open interval $(a,b)\subset[0,18]$ such that $(a,b)$ does not contain any roots of $Q\in A,$ $(a,b)\cap \Sigma_A' =\emptyset$,   and $\mu_A((a,b))\gg 1$. 
Note that $I(\mu_A)\geq 0$, so $\mu_A$ does not have any atom measure in its support. This and \eqref{zeroid}, implies that there exists some distinct $z_1,z_2,z_3 \in (a,b)\cap \Sigma_A $ such that 
\[F(z_i)=0.\]
By mean value theorem for the second derivative,  there exists $z_4\in (a,b)$ such that
\[
F''(z_4)=0.
\]
On the other hand, we have
\begin{equation}\label{double}
    F''(z)=\int \frac{1}{(z-x)^2}d\nu_A(x)+\sum_{Q\in A}\lambda_Q\sum_{\alpha,Q(\alpha)=0}\frac{1}{(z-\alpha)^2}\gg 1 .
\end{equation}
for every $z\in (a,b).$
This is a contradiction. Therefore, $\nu_A$ has the same support as $\mu_A$. 
\newline

By~\eqref{StrD} and~\eqref{Lag}, we have 
\begin{equation}
 \int\left| U_{\mu_A}(x)-\sum_{Q\in A} b_{Q}\log|Q(x)|\right|d\nu_A(x)+\sum_Q \lambda_Q\int \log|Q|d\mu_A(x)+b_Q\int\log|Q|d\nu_A(x)=0
\end{equation}
which implies that
\begin{equation}
    U_{\mu_A}(x)=\sum_{Q\in A} b_{Q}\log|Q(x)|
\end{equation} for every $x\in \Sigma_A$, 
\begin{equation}\label{mulogQ}
    \int \log|Q(x)|d\mu_A(x)=0
\end{equation}
for every $Q\in A$ where $\lambda_Q\neq 0$, and 
\begin{equation}
    \int\log|Q|d\nu_A(x)=0
\end{equation}
for every $Q\in A$ where $b_{Q}\neq 0$.
\newline

Next, we show that $\Sigma_A'=\Sigma_A$ is a finite union of intervals and every gap between the intervals contains a root of some $Q\in A$. Suppose to the contrary that there exists an open interval $(a,b)$ inside $[0,18]$ which does not contain any root of $Q$ for any $Q\in A$ such that $a,b\in \Sigma_A$ and $(a,b)\cap \Sigma_A=\emptyset.$ Let $F$ be the function that is defined in \eqref{F}. By the mean value theorem for second derivatives, we have
\[
\frac{F(b)+F(a)-2F(\frac{a+b}{2})}{(\frac{b-a}{2})^2}=\frac{-2F(\frac{a+b}{2})}{(\frac{b-a}{2})^2}=F''(z_0)
\]
for some $z_0\in (a,b)$. By~\eqref{double}, we have 
\[
\frac{-2F(\frac{a+b}{2})}{(\frac{b-a}{2})^2}=F''(z_0)>0
\]
This implies that $F(\frac{a+b}{2})<0$ which is a contradiction. By Proposition~\ref{analytic}, it follows that $\mu_{A,\Sigma_A}$ constructed in Proposition~\ref{analytic} is a solution to the primal problem. Similarly,  $\nu_A$ is the  measure with potential function constructed in \cite[section 3]{lowerbound}
\[
U_{\nu_{A}}(x)=x-\sum_{Q\in A}\lambda_Q\log|Q(x)|-\Lambda_A
\]
for $x\in \Sigma_A$.
This completes the proof of our proposition. 
\end{proof}

Finally, we give a proof of Theorem~\ref{main1}.

\section{Proofs of Theorem~\ref{main1} and Theorem~\ref{dual}}\label{main_thms}
In this section, we give an explicit formula for the density function of the optimal measures for the primal and the dual linear programming problems, and we complete the proofs of Theorem~\ref{main1} and Theorem~\ref{dual}. Recall our notation from the introduction. First, we give a proof of Theorem~\ref{main1}. 

\begin{proof}[Proof of Theorem~\ref{main1}]
    By Proposition~\ref{mainprop},  there is a  solution $(\mu_A,b_Q)\in \mathcal{C}(A)^+$ to the primal problem associated to $A$, where $\mu_A$ is supported on  a finite union of intervals $\Sigma_A=\bigcup_{i=0}^l[a_{2i},a_{2i+1}] $. Moreover, $(\mu_A,b_Q)$  satisfies all the  properties stated in Theorem~\ref{main1}  except proving the explicit formula for the density function of 
$\mu_A$. 
By Proposition~\ref{mainprop},  
   \begin{equation}
           U_{\mu_A}(x)=\sum_{Q\in A} b_Q \log|Q(x)|
   \end{equation}
    for every $x\in \Sigma_A$. Using the above, one can recover $\mu_A$ by solving the inverse log potential problem as in Proposition~\ref{analytic} and conclude that 
    $$d\mu_A(x)=\frac{|P(x)|}{\prod_{Q\in A}|Q(x)|\sqrt{|H(x)|}}dx$$
where $H(x)=\prod_{i=0}^{2l+1}(x-a_i)$ and  $P(x)$ is a polynomial of degree $\sum_{Q\in A}\deg(Q)+l$ if $\{x\}\subset A$.  

Finally, suppose  that  for every $Q\in A$ all roots are non-negative and $\Tr(Q)< \Lambda_A$. By Proposition~\ref{smooth}, $\lambda_Q>0$. By Proposition~\ref{mainprop}, we have
\[
\int \log|Q(x)|d\mu_A(x)=0
\]
for every $Q\in A$ and any solution $\mu$ of the primal problem.
By Proposition~\ref{mainprop},
\[
U_{\mu_A}(x)=\sum_{Q\in A} b_Q\log|Q(x)|
\]
which implies 
\[
I(\mu_A)=\int U_{\mu_A}(x) d\mu(x)=0
\]
The uniqueness follows from the concavity of the logarithmic energy. In fact, if we have two distinct primal solutions $\mu_{A,1}$ and $\mu_{A,2}$, then $\frac{\mu_{A,1}+\mu_{A,2}}{2}$ is also a solution to the primal problem with strictly positive logarithmic energy which contradicts with $I(\frac{\mu_{A,1}+\mu_{A,2}}{2})=0$. This completes the proof of our propostion.  
\end{proof}

Finally, we give a proof of Theorem~\ref{dual}. 

\begin{proof}[Proof of Theorem~\ref{dual}] By Proposition~\ref{mainprop}, 
    there exists  $(\nu_A,\lambda_{A,Q},\beta_{A,Q},\Lambda_{A})\in  \mathcal{D}(A)^+$  such that $\nu_A$ is supported on $\Sigma_A$ and satisfies all the properties of Theorem~\ref{dual} except proving the explicit formula for the density function of $\nu_A$. By Proposition~\ref{mainprop}  \[
    \begin{array}{cc}
                   x\geq \Lambda_A +\sum_{Q\in A}\lambda_{A,Q} \log|Q(x)|+ U_{\nu_A}(x),\text{ and}\\
          \beta_{A,Q}=\int \log|Q(x)| d \nu_A=0 \text{ if $b_Q\neq 0$,}

    \end{array}
    \]
where in the first line the inequality holds for non-negative $x$ and equality holds for every $x\in \Sigma_A$ with some scalars $\lambda_{A,Q}\geq 0$. By \cite[Proposition 4.1]{lowerbound}, the density function of $\nu_A$ vanishes on the positive boundary points of $\Sigma_A$. By a similar argument as in Proposition~\ref{analytic} and computing the top coefficient of the polynomial in the numerator of $\nu_A$, it follows that $\nu_A$ has the following density function 
\[
\frac{|p(x)|\sqrt{|H(x)|}}{\pi \prod_{\lambda_Q\neq 0}|Q(x)|},
\]
where $H(x)=\prod_{i=0}^{2l+1}(x-a_i)$ and $p(x)$ is a monic polynomial with degree $\sum_{Q\in A} \deg(Q)-l-1$.

Finally, we show that  $I(\nu_A)< 0$. By concavity of the logarithmic energy~\cite[(6)]{lowerbound} and since $\nu_A\neq \mu_A$, we have
\[
0=2I(\mu_A,\frac{\nu_A}{\nu_A(\mathbb{R})})> I(\mu_A)+\frac{I(\nu_A)}{\nu_A(\mathbb{R})^2}.
\]
Since 
\[
I(\mu_A)=\sum_{Q} b_Q \int \log|Q(x)|d\mu_A(x)\geq 0. 
\]
This implies that 
\[
I(\nu_A)<0.
\]
Let $\mu$ be the arithmetic measure supported on non-negative real numbers with \( \int x d\mu(x)=\lambda^{SSS}\) that we constructed in Corollary~\ref{Acor}. 
We take the average of \[
x\geq \Lambda_A +\sum_{Q\in A}\lambda_{A,Q} \log|Q(x)|+ U_{\nu_A}(x)
\] with respect to $\mu$ and obtain  
\[
\lambda^{SSS}\geq \Lambda_A+ \nu_A(\mathbb{R})I(\frac{\nu_A}{\nu_A(\mathbb{R})},\mu)\geq \Lambda_A+\frac{I(\nu_A)}{2\nu_A(\mathbb{R})}.
\]
where we used the fact that 
$\int \log|Q| d\mu(x)\geq 0$ and the 
 concavity of the logarithmic energy. This completes the proof of Theorem~\ref{dual}.
\end{proof}

\section{Explicit Values for Serre's contruction}\label{serre_proof}
In this section, we find expressions for $\Lambda_\emptyset$ and $\Lambda_{\{x\}}$.
This heavily uses the computations done in Appendix B of \cite{lowerbound} which shows how to compute density functions of measures with potentials of the form $ax+b\log(x)+c$.
\begin{proof}[Proof of Corollary \ref{schur_bound}]
By Theorem \ref{main1}, $U_{\mu_A}(x)$ vanishes on the support $[a,b]$ of $\mu_A$ and is positive on the complement. Therefore $b-a=4$ and $\mu_A$ is the equilibrium measure on $[a,b]$. Minimizing the trace, we get $[a,b]=[0,4]$, $d\mu_A=\frac{dx}{\pi\sqrt{x(4-x)}}$, and $\int xd\mu_A=2$.
From equation (32) of \cite{lowerbound}, $d\nu_A = \frac{dx}{\pi\sqrt{x(4-x)}} + \frac{(2-x)dx}{2\pi\sqrt{x(4-x)}}=\frac{(4-x)dx}{2\pi\sqrt{x(4-x)}}$ as desired with potential function $\frac{1}{2}x-1$ on $[0,4]$.
\end{proof}

\begin{lemma}
Let $(\hat x, \hat y)$ be the unique solution to
\[\begin{cases}
y\log y = (y-x)\log(y-x)+x\\
\log(x)\log(y)=\log(xy)\log(y-x).
\end{cases}\]
Then $\Lambda_A = \hat y - \hat y \log\hat y + \hat x - \hat x\log\hat x$.
\end{lemma}
Here if the support is $\Sigma_A=[a,b]$, $\hat x$ represents $\frac{a+b-2\sqrt{ab}}{4}$ and $\hat y$ represents $\frac{a+b+2\sqrt{ab}}{4}$. Note that the approximating support given by Serre as given in Corollary \ref{serre_bound} satisfies the system of equations with error at most $10^{-10}$ and giving $\Lambda_{\{x\}}\in(1.8983020088,1.8983020091)$.
\begin{proof}
    As seen by equation (33) in \cite{lowerbound}, the fact that $\int \log|x|d\nu_A=0 $ is equivalent to
    \begin{equation}\label{log_zero_serre} y\log y - y = (y-x)\log(y-x)-(y-x)\end{equation}
    where $y=\int zd\nu_A(z)=\frac{a+b+2\sqrt{ab}}{4}$ and $x=\frac{a+b-2\sqrt{ab}}{4}$ for $0<a<b$. For the other equation, we look at the primal problem.
    \newline

    By Theorem \ref{main1}, we have that $U_{\mu_A}(x)=c\log(x)$ on $[a,b]$ for some positive constant $c$. The following is equivalent to the constant term of $U_{\mu_A}(x)$ being zero:
    \begin{align*}
    0&=(1-c)\log\left(\frac{b-a}{4}\right) - c \log\frac{a+b+2\sqrt{ab}}{b-a}\\
    &= \log\left(\frac{b-a}{4}\right)-c\log\frac{a+b+2\sqrt{ab}}{4}\\
    &= \frac{1}{2}\log(xy)-c\log(y).
    \end{align*}
    So that $c=\frac{\log(xy)}{2\log(y)}$.
    We are asked to minimize expectation of $\mu$ which is
    \begin{equation}\label{expectation_serre}(y-x)+2(1-c)x=y+x-\frac{x\log(xy)}{\log(y)}=y-\frac{x\log(x)}{\log(y)}.\end{equation}
    The following is equivalent to $\int \log|x|d\mu_A(x)=0$:
    \[
        0=\frac{1}{2}\log\left(xy\right)-2c^2\log\left(\frac{y}{y-x}\right).\]
    Since $c=\frac{\log(xy)}{2\log(y)}$, we have
    $$0=\frac{1}{2}\log(xy)-\frac{\log(xy)^2}{2\log(y)^2}\left(\log(y)-\log(y-x)\right).$$
    Dividing both sides by $\frac{\log(xy)}{2\log(y)}$, we can rearrange to get that
    \begin{equation}\label{energy_zero_serre}
        \log(x)\log(y)=\log(xy)\log(y-x).
    \end{equation}
    Thus for some solution $(\hat x, \hat y)$ to the simultaneous system \eqref{log_zero_serre} and \eqref{energy_zero_serre}, $\Lambda_A$ is given by $\hat y - \frac{\hat x\log\hat x}{\log \hat y}$ as stated in \eqref{expectation_serre}.
    From \eqref{log_zero_serre}, we can multiply by $\log(x)$ to get
    $$x\log x=y\log(y)\log(x) -(y-x)\log(y-x)\log(x).$$
    Using \eqref{energy_zero_serre} and simplifying, we get
    $$x\log x = (y\log(y)+x\log(x))\log(y-x).$$
    Thus
    $$\Lambda_A=\hat y - \frac{\hat x\log\hat x}{\log\hat y}=\hat y - \frac{(\hat y\log\hat y+\hat x\log\hat x)\log(\hat y - \hat x)}{\log\hat y}= \hat y - (\hat y - \hat x)\log(\hat y - \hat x)-\frac{\hat x\log(\hat x\hat y)\log(\hat y - \hat x)}{\log\hat y}.$$
    Applying \eqref{log_zero_serre} and \eqref{energy_zero_serre} one last time, we have
    $\Lambda_A = \hat y - \hat y\log\hat y + \hat x - \hat x \log\hat x$.
    It remains to show that the system of equations given by \eqref{log_zero_serre} and \eqref{energy_zero_serre} has a unique solution.
    \newline

    The fact that the system has a solution is a consequence of Theorem \ref{main1}. For uniqueness, it suffices to show by uniqueness of the optimal dual measure per Theorem \ref{dual} that $\lambda_A$ given by the dual measure is also $\hat y - \hat y\log\hat y + \hat x - \hat x \log\hat x$ and that any solutions to the system of equations correspond to positive measures.
    The value of $\lambda_A$ is $\int zd\nu_A(z)-2 xI_\nu$ assuming the density function of $\nu_A$ vanishes at $a$ and $b$ per Theorem \ref{dual}. This is
    $$ y- 2x\left(\frac{1}{2}\log(xy)-\frac{2g^2}{(m-g)^2}\log\frac{m+g}{2g}+\frac{3g-m}{2(m-g)}\right)$$
    where $m=x+y$ and $g=y-x$. This simplifies to
    $$2x- x\log(xy)+\frac{(y-x)^2}{x}\log\frac{y}{y-x}.$$
    Since $(y-x)\log(y-x)=y\log y -x$, after simplifying we get
    $$ x-x\log(x)+y-y\log(y).$$
    Thus the lower bound on $\Lambda_A$ given by the dual measure matches the upper bound on $\Lambda_A$ given by the primal solution.
    To complete the proof of uniqueness, we need to show that solutions to
    \[\begin{cases}
y\log y = (y-x)\log(y-x)+x\\
\log(x)\log(y)=\log(xy)\log(y-x)
\end{cases}\]
    only correspond to {\it positive} measures in the primal and dual problems. For the dual problem, the density functions are always of the form $\hat c \sqrt{(b-x)(x-a)}$ for some $\hat c>0$ and $0<a<b$. Hence these measures are positive. For the primal problem, it suffices to show $c=\frac{\log(\hat x\hat y)}{2\log(\hat y)}>0$. We show the more general statement that $xy>1$ and $y>1$ for any point $(x,y)$ on $\log(x)\log(y)=\log(xy)\log(y-x)$.
    \newline

    Let $L(x,y)=\log(x)\log(y)$ and $R(x,y) = \log(xy)\log(y-x)$. Suppose $L(x,y)=R(x,y)$. Firstly, $y>x>0$ in order for $L$ and $R$ to be defined.
    Furthermore, for the sake of contradiction assume $y\le 1$. Then we have $\log(xy)\le \log(x)< 0$ and $\log(y-x)<\log(y)\le0$ which is a contradiction because then $R(x,y)>L(x,y)$. So $y>1$.
    Now assume $xy\le 1$ so that $0<x<1<y$ and $xy\le 1$. Firstly, $xy\neq 1$ because then $R(x,y)=0$ and $L(x,y)<0$. If $xy<1$, then $y-x>1$ because otherwise $L(x,y)<0$ and $R(x,y)\ge0$. Therefore,
    $$0 < -\log(xy)<-\log(x)$$
    and
    $$0 < \log(y-x) < \log(y).$$
    Multiplying, we have that $R(x,y)>L(x,y)$ which is again a contradiction. Therefore, $xy>1$ and $y>1$ for any solution to the system of equations. This completes the proof.
\end{proof}

\section{Numerical approximations of $\mu_A$ and $\nu_A$}\label{validation}
Fix a finite subset of integral polynomials $A\subset \mathbb{Z}[x]$ with all real roots. Recall our notations from Theorem~\ref{main1} and Theorem~\ref{dual}. Let $\Sigma_A:=\bigcup_{i=0}^l[a_{2i},a_{2i+1}]$ be the support of the optimal measures $\mu_A$ and $\nu_A$. 
In this section, we  explain how to approximate to $\nu_A$ and $\mu_A$ numerically. First, we express $\mu_A$ and $\nu_A$ explicitly in terms of their support $\Sigma_A$. Then we apply our gradient descent algorithm to approximate $\Sigma_A$ numerically.
\subsection{Computing $\nu_A$}\label{Xdef}
We proved in Theorem~\ref{dual} that for every $x\in \Sigma_A$
\[
x= \Lambda_A +\sum_{Q\in A}\lambda_{Q} \log|Q(x)|+ U_{\nu_A}(x)
\]
and that 
\[
d\nu_A(x)=c\frac{|p(x)|\sqrt{|H(x)|}}{\prod_{\lambda_Q\neq 0}|Q(x)|}.
\]
Let $\mu_{eq}$, $\mu_{lin}$, $\mu_Q$ be  the equilibrium measure, the measure with linear potential $U_{\mu_{lin}}(t)=t$, and the probability measure with potential $U_{\mu_{Q}}(t)=\frac{\log(|Q(t)|)}{\deg(Q)}+C_Q$ for some constant $C_Q$ supported on $\Sigma_A$ respectively. By \cite[Lemma~3.2]{lowerbound}, 
the density function of  $\mu_Q$ is given by the following explicit expression:
\[
d\mu_Q(x)=\frac{1}{\deg(Q)}\sum_{Q(\alpha)=0}\frac{|P_{\alpha}(x-\alpha)|\sqrt{H(\alpha)}}{|x-\alpha|\pi\sqrt{H(x)}},
\]
where $P(x):=x^{\deg(P_{\alpha})}p_\alpha(1/x)$ is the polynomial that appears in the numerator of the equilibrium measure of the  image of $\Sigma_A$ under the map $z\to \frac{1}{z-\alpha}$. By Theorem~\ref{dual} and the above
\[
\begin{split}
X_{eq}\frac{|p_{eq}(x)|}{\pi\sqrt{H(x)}}+\frac{-x|p_{eq}(x)|+|P_{gap}(x)|}{\pi\sqrt{H(x)}}+\sum_{Q}\frac{X_Q}{\deg(Q)}\sum_{Q(\alpha)=0}\frac{|P_{\alpha}(x-\alpha)|\sqrt{H(\alpha)}}{|x-\alpha|\pi\sqrt{H(x)}}    \\=c\frac{|p(x)|\sqrt{|H(x)|}}{\prod_{\lambda_Q\neq 0}|Q(x)|},
\end{split}
\]
where $X_{eq}>0,$ and $X_Q<0$, $c>0$ is a normalizing constant, and $\tilde{A}=\{Q\in A: X_{Q}\neq 0\}$. 
From this, it is then derived that
\begin{equation}\label{residue}
X_Q=-\frac{\deg(Q)c\pi |\text{Res}(p,Q)|^{1/\deg(Q)}\sqrt{|\text{Res}(H,Q)|^{1/\deg(Q)}}}{|\text{Disc}(Q)|^{1/\deg(Q)}|\prod_{Q_i\neq Q}|\text{Res}(Q,Q_i)|^{1/\deg(Q)}|}.
\end{equation}
Suppose that we have exactly one root of $Q\in A$ inside every gap, then $\deg(p)=0$ and up to a scalar $\nu_A$ has the following density function
\[
\frac{\sqrt{|H(x)|}}{\prod_{\lambda_Q\neq 0}|Q(x)|}.
\]
Therefore, assuming we have exactly one root of $Q\in A$ inside every gap, we obtain \[X_Q=-\frac{c\pi\deg(Q)|\text{Res}(H,Q)|^\frac{1}{2\deg(Q)}}{|\text{Disc}(Q)|^\frac{1}{\deg(Q)}}\] for $Q\in A$ since $|\text{Res}(P,Q)|=1$ for any $P\neq Q$ in 
$A=\{x, x-1, x^2-3x+1, x^3-5x^2+6x-1 
\}$. 
 \subsection{Computing $\mu_A$}
 Recall the definition of $\mu_{eq}$ and $\mu_Q$ from the previous subsection.
We proved in Theorem~\ref{main1} that 
\[
U_{\mu_A}(x)=\sum_{Q\in A}b_Q\log|Q(x)|
\]
for every $x\in \Sigma_A$. It follows that
\[
\mu_A=\left(1-\sum_{Q\in A} b_Q\right)\mu_{eq}+\sum_{Q\in A}b_Q\mu_{Q}.
\]
Moreover, by \eqref{mulogQ}, if $\lambda_Q\neq 0$ then
\[
\int \log|Q(x)|d\mu_A(x)=0.
\]
By assuming $\lambda_Q\neq 0$ for every $Q\in A$, it follows that $b_Q$ satisfies the following system of linear equations:
\[
\sum_{Q} b_QW(Q',Q)=W(Q')
\]
where 
\[
W(Q',Q):=\int \log|Q'(x)|d\mu_{Q}(x)-\int \log|Q'(x)|d\mu_{eq}(x)
\]
and 
\[
W(Q')=\int \log|Q'(x)|d\mu_{eq}(x).
\]
Assuming $\det(W(Q,Q'))\neq 0$, one can write
\begin{equation}\label{[bQ]}
  [b_Q]_{Q\in A}= [W(Q,Q')]_{Q,Q'\in A}^{-1}\times[W(Q')]_{Q'\in A}.  
\end{equation}
Note $\Tr(Q)<\lambda_{SSS}$ for every $Q\in \{x, x-1, x^2-3x+1, x^3-5x^2+6x-1\}$. By Proposition~\ref{smooth}, $\lambda_Q\neq 0$ for every $Q\in \{x, x-1, x^2-3x+1, x^3-5x^2+6x-1\}$.  Our numerical calculation also shows that $\det(W(Q,Q'))\neq 0$ for 
$Q,Q'\in\{x, x-1, x^2-3x+1, x^3-5x^2+6x-1\}.
$ 
\subsection{Gradient descent algorithm}\label{grad_desc}
We estimate $\Lambda_A$ by adapting the gradient descent algorithm developed in \cite{lowerbound}. Given $\Sigma=\bigcup_i[a_{2i}, a_{2i+1}]$, we are able to compute $X_{lin}$, $X_{eq}$ and $X_Q$ for each $Q\in A$ as in section \ref{Xdef}.
We then compute $\mu_{eq, \Sigma}$, $\mu_{lin, \Sigma}$, and $\mu_{Q,\Sigma}$ for each $Q\in A$ as shown in section \ref{Xdef}. Define
$$\nu_\Sigma := X_{eq}\mu_{eq,\Sigma}+X_{lin,\Sigma}\mu_{lin,\Sigma}+\sum_{Q\in A}X_Q\mu_{Q,\Sigma}.$$
We then compute the matrix $[W(Q',Q)]$ and compute $b_Q$ coefficients as in \eqref{[bQ]}, and define
\[
\mu_{\Sigma}:=\left(1-\sum_{Q\in A} b_Q\right)\mu_{eq}+\sum_{Q\in A}b_Q\mu_{Q}.
\]

With these density function, we can compute all of the required quantities of a particular measure. So the high-level algorithm is to perform gradient descent on $\Sigma$ numerically (where $a_i$'s are the parameters) minimizing the sum of squares of  $I(\nu_\Sigma,\mu_{\Sigma})$, the value of the density function of $\nu_{\Sigma}$ at the $a_i$'s, the values of $\int\log|Q(x)|d\mu_\Sigma$ and $\int\log|Q(x)|d\nu_\Sigma$ for $Q\in A$, and finally $\int xd\mu_{\Sigma}-\Lambda_{\Sigma}$ because these values are 0 from Theorem~\ref{main1} and Theorem~\ref{dual}.
\newline

We use the same numerical methods that we introduce in~\cite{lowerbound}. We refer the reader to \cite[Section 6.2]{lowerbound} for further details of the numerical computations of singular integrals and the optimization methods that we introduce to increase the speed of the computations. 

\subsection{Verification of Results}    
In this section, we verify that given numerical approximations to $\nu_A$ and $\mu_A$, we get valid lower and upper bounds on $\Lambda_A$, respectively.
Fix a finite subset $A\subset\mathbb{Z}[x]$ and $\Sigma=\bigcup_{i=0}^l[a_{2i},a_{2i+1}]$. Suppose $\mu$ and $\nu$ are Borel probability measures with support $\Sigma$.
Let $b_Q'\ge 0$ for each $Q\in A$ so that $U_\mu(x)=\sum_{Q\in A}b_Q'\log|Q(x)|$ for each $x\in \Sigma$ and let  $\lambda_0'>0$ and $\lambda_Q'\ge 0$ for all $Q\in A$ such that $x = \lambda_0' + \sum_{Q\in A} \lambda_Q' \log|Q(x)| + U_\nu(x)$ for each $x\in \Sigma$.
Further assume that $\int \log|Q(x)|d\mu\ge 0$ and $\int \log|Q(x)|d\nu\ge 0$ for $Q\in A$.
In the rest of this section, we prove that $\lambda_0' -\delta\log(18) \le \Lambda_A\le \int xd\mu$.
In particular, we must also show that
\begin{equation}\label{mu_ineq}    
U_\mu(x)\ge \sum_{Q\in A}b_Q'\log|Q(x)|\text{ for all }x\in \mathbb{R}^+
\end{equation}
and
\begin{equation}\label{nu_ineq}
x \ge \lambda_0' - \delta\log(18) + \sum_{Q\in A} \lambda_Q' \log|Q(x)| +U_\nu(x)\text{ for all }x\in \mathbb{R}^+
\end{equation}
given equality on $\Sigma$. We do this for special classes of $A$ and $\mu$ which are of interest to us and which are satisfied by Corollaries \ref{first_numer_approx} through \ref{our_bound}. The following proposition gives a lower bound estimate on $\Lambda_A$. The proof is based off of the proof of Proposition 6.1 in \cite{lowerbound}.
\begin{proposition}\label{eq_sufficient}
    Let $A, \Sigma, \lambda_0',$ and $\lambda_Q'$ for all $Q\in A$ be defined as above. Suppose $\Sigma\subset [0,18]$ and $\mu_{eq}$ is the equilibrium measure of $\Sigma$. Define $\mathcal R$ to be $\{\alpha\in\mathbb R: \exists Q\in A,Q(\alpha)=0\}$.
    Suppose that $\mathcal R\cap (a_{2i+1},a_{2(i+1)})=\{r_i\}$ for each $i=0,\dots, l-1$ and $\mathcal R\cap \big([0,a_0)\cup (a_{2l+1},\infty)\big)=\{0\}$. Fix $\delta>0$ so that for each $a_i$,
    \[
    \delta_{i}:=\lim_{\stackrel{x\to a_i}{x\in \Sigma}} \frac{d\nu(x)}{d\mu_{eq}(x)}\leq \delta, 
    \]
     for any $x\in (a_{2i+1},r_i)$,
    \[
    \sum_{Q\in A}\sum_{Q(\alpha)=0}\frac{\lambda_Q'}{(x-\alpha)^2} + \int \frac{f_{2i+1}(y)}{(x-y)^2}dy>0
    \]
     where $f_{i}(y)$ is the density function of $\nu-\delta_{i}\mu_{eq}$, and for any $x\in (r_i,a_{2(i+1)})$,
    \[
    \sum_{Q\in A}\sum_{Q(\alpha)=0}\frac{\lambda_Q'}{(x-\alpha)^2} + \int \frac{f_{2(i+1)}(y)}{(x-y)^2}dy>0.
    \]
    Then $\lambda_0'-\delta \log(18)\le\Lambda_A$.
\end{proposition}
\begin{proof}
Let $\mu_A$ is the measure given by Theorem~\ref{main1} with support  $\Sigma_A\subset[0,18]$.

    It suffices to prove 
    \[
    x \ge \lambda_0'-\delta \log(18) + \sum_{Q\in A} \lambda_Q' \log|Q(x)| +U_\nu(x)\text{ for all }x\in [0,18].
    \]
    Indeed, integrating both sides with respect to $d\mu_A(x)$, we conclude
    \[
    \Lambda_A\geq \lambda_0'-\delta \log(18) +I(\mu_A,\nu).
    \]
    Note that $I(\mu_A,\nu)=\int U_{\mu_A}(x)d\nu\ge \int \sum_{Q\in A}b_Q\log(x)d\nu\ge 0$ so that $\Lambda_A \ge \lambda_0'-\delta\log(18)$.
    Let 
    $$g(x) := x - \lambda_0'+\delta\log(18) - \sum_{Q\in A} \lambda_Q' \log|Q(x)| -U_\nu(x).$$ 
    It is sufficient to show positivity of $g(x)\ge 0$ on $[0,18]$. By assumption, $g(x)= \delta \log(18) > 0$ on $\Sigma$ so we must show that this holds on the complement. We start by showing $g(x)\geq 0$ for $x\in (a_{2i+1},r_i)$. The proof for the remaining intervals is similar. 
    \newline
    
    By the assumptions of the proposition, the density of $\nu-\delta_{2i+1}\mu_{eq}$ vanishes at $a_{2i+1}$ for some $\delta_{2i+1}<\delta$. Let 
     \[
    f(x) := x - \lambda_0'- \sum_{Q\in A} \lambda_Q' \log|Q(x)| - U_{\nu-\delta_{2i+1}\mu_{eq}}(x).
    \]
Note that $f(x)=\delta_{2i+1} C_{\Sigma}\geq 0$ on $\Sigma$, where $C_\Sigma>0$ is the logarithmic capacity of $\Sigma$  and
\[
g(x)-f(x)=\delta \log(18)-\delta_{2i+1}U_{\mu_{eq}}(x)\geq 0
\]
for $0\le x\le 18$.
It is enough to show that $f(x)\geq 0$ for any $x\in (a_{2i+1},r_i).$  This is deduced from the fact that for $x\in \mathbb R^+\setminus\Sigma$,
    $$f''(x) = \sum_{Q\in A}\sum_{Q(\alpha)=0}\frac{\lambda_Q'}{(x-\alpha)^2} + \int \frac{f_{2i+1}(y)}{(x-y)^2}dy>0,$$
     where $f_{2i+1}(y)$ is the density function of $\nu-\delta_{2i+1}\mu_{eq}$, which is positive  by our assumption. We get this by Liebnitz's rule since $\mathbb R^+\setminus\Sigma$ is open.
    By Lemma 4.2 of \cite{lowerbound}, $f'(x)=0$ at $a_{2i+1}$. Since we assumed that there is only one root between any two intervals in $\Sigma$, $g(x)\ge 0 $ on $\mathbb R^+$.
\end{proof}
Now let $\mu_Q$ denote the unique probability measure with potential $U_{\mu_Q}(x)\ge \frac{\log|Q(x)|}{\deg(Q)}-C_{Q,\epsilon}$ with equality precisely on $\Sigma$ for some constant $C_{Q,\epsilon}>0$. Since in the gradient descent algorithm, we choose $\mu=c\mu_{eq}+\sum_{Q\in A}b_Q'\deg(Q)\mu_Q$
for suitable $c>0$, it follows that $U_\mu(x) \ge \sum_{Q\in A}b_Q'\log|Q(x)|$ with equality precisely on $\Sigma$.
Therefore, it simply remains to check that the density function of $\nu$ is positive which is possible because the numerator of its density function is a polynomial combination of absolute values of polynomials.
Once this is established, $\lambda_0'-\delta\log(18)\le \Lambda_A\le \int xd\mu$.

\subsection{Our Bound}\label{our_bound_proof}
Here we give details regarding the measures described in Corollary \ref{our_bound}. This numerical result is concerning $A=\{Q_0,Q_1,Q_2,Q_3\}$ where $Q_0:=x,Q_1:=x-1,Q_2:=x^2-3x+1,$ and $Q_3:=x^{3}-5x^{2}+6x-1$.
The proofs of Corollaries~\ref{first_numer_approx} and \ref{approx_quad} are similar. For these, the supports of the measures used  are provided in the statement of the corollaries which is sufficient to reproduce the given bounds by following the same procedure presented below.
\begin{proof}[Proof of Corollary~\ref{our_bound}]
Let  $\Sigma:=\bigcup_{k=0}^6[a_{2k},a_{2k+1}],$ where
\[
\begin{split}
[a_0,\dots,a_{13}]=[0.06129179514230616, 0.17515032062914598,\\
0w.22234021363650555, 0.31342480981675824, 0.45755305611054947,\\
0.7464580569205782, 1.3003810514246414, 1.4785537920437593,\\
1.629639282543484, 2.391040662791468, 2.8714421221120547, \\
3.108899604515233, 3.387062952658943, 4.905820899328081]    
\end{split}
\]

Using this, we compute $H(x) = -\prod_{k=0}^{14}(x-a_k)$.
This is sufficient to compute the equilibrium measure of $\Sigma$ which is given by
\[
d\mu_{eq}=\frac{|P_{eq}(x)|}{\pi \sqrt{|H(x)|}}
\]
where 
\[
\begin{split}
  P_{eq}(x)\approx
-1.01331193 + 10.12197831 x - 33.33843935 x^2 \\
+ 46.34630353 x^3 - 30.14669018 x^4 + 9.02111358 x^5 - x^6.
\end{split}
\]

Now we start computing details of the dual measure. We normalize the dual measure by a constant factor so that it is a probability measure on $\Sigma$. The normalizing coefficient $c$ via
\[
\int_{\Sigma} \frac{c\sqrt{H(x)}}{\prod_{Q\in \tilde{A}} |Q(x)|}dx=1
\]
which yields $c\approx 0.39829154.$
We know that $X_{lin}$ is therefore
\[X_{lin} = c\pi \approx 1.25126977. \]
From \eqref{residue}, we choose 
\[
X_{Q_0}= -0.66676107\approx  -c\pi|\res(Q_0,H)|^{1/2},
\]
\[
X_{Q_1}=-0.64219523\approx -c\pi|\res(Q_1,H)|^{1/2},
\]
\[
X_{Q_2}=-0.52351453\approx-\frac{2c\pi|\res(Q_2,H)|^{1/4}}{\sqrt{5}},
\]
\[
X_{Q_3}=-0.3263993\approx-\frac{3c\pi|\res(Q_3,H)|^{1/6}}{7^{2/3}}.
\]
In order that $\nu_A$ be a probability measure, we let 
\[
X_{eq}\approx 6.29132246.
\]
However, since this is only an approximation of the goal measure, this does not have all of the positive values we require (including positivity of the density function). To account for this, we take
\[
\nu:=2(10)^{-8}\mu_{eq} + \frac{1-2(10)^{-8}}{X_{eq}+X_{lin}\mu_{lin}(\mathbb{R})+\sum_{i=0}^3X_{Q_i}}\left(X_{eq}\mu_{eq}+\sum_{i=0}^3X_Q\mu_Q+X_{lin}\mu_{lin}\right).
\]
For simplicity, we rename the new coefficients of $\mu_Q$ to $X_Q$, the new coefficient of $\mu_{lin}$ to be $X_{lin}$, and similarly for that of $\mu_{eq}$.
One can check that  for $x\in \Sigma$,
\[
x =  \lambda - \sum_{Q\in A} \frac{X_Q}{X_{lin}}\log|Q(x)| - \frac{1}{X_{lin}}U_\nu(x)
\]
for every $x\in \Sigma$
where 
\[
\lambda \geq 1.8215998.
\]
To get a lower bound on $\Lambda_A$, we use Proposition~\ref{eq_sufficient}. Let
 \[
    \delta_{i}:=\lim_{x\to a_i,x\in \Sigma} \frac{d\nu(x)}{d\mu_{eq}(x)}=\frac{|P(a_i)|}{|P_{eq}(a_i)\prod_{Q\in A}Q(a_i)|}.
    \]
One can check that $\delta\leq 3(10)^{-8}$.
Moreover, the statement that
    \[
   \int \frac{f_{2i}(y)}{(x-y)^2}dy>0
    \]
    holds in a neighborhood of $a_{2i}$ is equivalent to
    \[
    \frac{d}{dx}\frac{|P(x)|}{\prod_Q |Q(x)|}|_{x=a_{2i}}\geq \frac{d}{dx} \big|P_{eq}(x)|_{x=a_{2i}}.
    \]
    Using $\nu$ as described above, we can compute that
    \[
    \begin{split}
    \left[\frac{d}{dx}\frac{|P(x)|}{\prod_Q |Q(x)|}|_{x=a_{2i}}: 0\leq i\leq 6\right]\approx[45, 8, 7, 5, 10, 24, 100]
    \\
\left[\frac{d}{dx} |P_{eq}(x)|_{x=a_{2i}}: 0\leq i\leq 6\right]\approx[-7, 1, 1, 0.03, 2, 5, 45]
\end{split}
    \]
    which confirms half of the inequalities required by Proposition~\ref{eq_sufficient} near the boundary points.
    Similarly, the statement that
    \[
   \int \frac{f_{2i+1}(y)}{(x-y)^2}dy>0
    \]
    holds in a neighborhood of $a_{2i+1}$ is equivalent to
    \[
    \frac{d}{dx}\frac{|P(x)|}{\prod_Q |Q(x)|}|_{x=a_{2i+1}}\leq \frac{d}{dx} |P_{eq}(x)|_{x=a_{2i+1}}.
    \]
 To verify the remaining requirements of Proposition~\ref{eq_sufficient} near the boundary points, we compute
    \[
    \begin{split}
    \left[\frac{d}{dx}\frac{|P(x)|}{\prod_Q |Q(x)|}|_{x=a_{2i+1}}: 0\leq i\leq 6\right]\approx[-12, -5, -7, -6, -25, -38, -1242]
    \\
\left[\frac{d}{dx} |P_{eq}(x)|_{x=a_{2i+1}}: 0\leq i\leq 6\right]\approx[-2, -0.4, -0.3, -1.1, -2.6, -5.7, 2131].
\end{split}
    \]
    For values of $x$ away from these neighborhoods, one can check the required inequalities numerically. 
By Proposition~\ref{eq_sufficient}, it follows that
\[
\Lambda_A\geq 1.8215997. 
\]
For the corresponding primal measure, if $\mu=Y_0\mu_{eq}+\sum_{Q\in A}Y_Q \mu_Q$, we compute the $Y_Q$'s as suggested in Section \ref{grad_desc}. 
With this,
$$[Y_0,Y_{Q_0},Y_{Q_1},Y_{Q_2},Y_{Q_3}]\approx [0.40812097, 0.09176568, 0.15086476, 0.207612  , 0.14163659].$$
Taking $\tilde\mu:=6(10)^{-8}\mu_{eq}+(1-6(10)^{-8})\mu$, its potential is greater than $\sum_{i=0}^3Y_{Q_i}\log|Q_i(x)|$.
From this,
\[
\int xd\tilde\mu\leq1.8215999.
\]
This completes the proof of  Corollary~\ref{our_bound}.
\end{proof}

\section{Open Questions}\label{conjectures}
The method provided herein helps us to compute better upper bounds on $\lambda^{SSS}$ from the Schur-Siegel-Smyth trace problem; however, many questions about the problem remain open.
Some such questions about the general Schur-Siegel-Smyth trace problem are on the uniqueness of an optimal distribution and properties of the optimal distribution.
\begin{conjecture}\label{unique_opt_dist}
    There is a unique arithmetic probability distribution $\mu_{SSS}$ compactly supported in $\mathbb R^+$ so that $\int xd\mu_{SSS}=\lambda^{SSS}$.
\end{conjecture}

\begin{conjecture}\label{opt_dist_zero_energy}
    If $\mu$ is an arithmetic probability distribution compactly supported in $\mathbb R^+$ such that $\int xd\mu=\lambda^{SSS}$, then $I(\mu)=0$.
\end{conjecture}
Note that Conjecture \ref{opt_dist_zero_energy} implies Conjecture \ref{unique_opt_dist} by strict concavity of the energy functional.
Note that Theorem~\ref{main1} proves a finite version of Conjecture~\ref{opt_dist_zero_energy}. 
Of course there is also future work in improving the bounds on $\lambda^{SSS}$ to a more precise value, but it is currently thought to be out of reach to compute an explicit form for $\lambda^{SSS}$.
\newline

With the creation of a new method, other open questions arise about the method itself. In particular, we are interested in properties of the measures provided by the method for fixed $A\subset\mathbb Z[x]$. Let $\eta_A$ denote the lower bound measure given in \cite[Theorem 1.1]{lowerbound}.
\begin{conjecture}\label{supp_containment}
    Let $A$ and $B$ be finite sets of integer polynomials with all of their roots being real. If $A\subset B$, then
    $$\text{supp}(\eta_B)\subset\text{supp}(\eta_A)$$
    and
    $$\text{supp}(\mu_B)\subset \text{supp}(\mu_A)\subset \text{supp}(\eta_A)$$
    where $\text{supp}(\mu)$ for a measure $\mu$ denotes the support of $\mu$.
\end{conjecture}
These containment proposed in Conjecture \ref{supp_containment} have been observed by the authors in the cases which are computed and given herein and in \cite{lowerbound}.
Another property studied of these measures is their logarithmic energies. We know that $I(\eta_A)=I(\mu_A)=0$, but $I(\nu_A)<0$.
\begin{conjecture}\label{Inu=0}
    Let $A_0\subset A_1\subset\dots$ be an increasing chain of finite subsets of $\mathbb Z[x]$ so that $\displaystyle\lim_{n\to\infty} \Lambda_{A_n}\to \lambda^{SSS}$. Then 
    $\frac{I(\nu_{A_n})}{\nu_{A_n}(\mathbb{R})}\le \frac{I(\nu_{A_{n+1}})}{\nu_{A_{n+1}}(\mathbb{R})}$ for all $n\in\mathbb N$ and
    $\displaystyle \lim_{n\to\infty}\frac{I(\nu_{A_n})}{\nu_{A_n}(\mathbb{R})}=0$.
\end{conjecture}

Another open problem is in regards to the polynomial $p(x)$ in Theorem \ref{dual}. This polynomial appears when multiple roots from the polynomials in $A$ occur in a single gap of the measure's support $\Sigma_A$. 
\begin{conjecture}
    Suppose that for a finite set $A\subset\mathbb Z[x]$ containing only polynomials with all real roots. Let $\Sigma_A=\bigcup_{i=0}^l[a_{2i},a_{2i+1}]$. Enumerate the roots of the polynomials in $A$ inside $(a_{2j+1},a_{2j+2})$ by $r_1<\dots<r_k$ for some $0\le j<l$ and $k>1$. Then the polynomial $p(x)$ from Theorem \ref{dual} has a unique root in each of the intervals $(r_s,r_{s+1})$ for $1\le s<k$.
\end{conjecture}
Note that this accounts for all of the roots of $p$. The motivation behind this conjecture is that there ``ought to be" an interval between the roots of the $Q$'s; however, these intervals are degenerate, one-point, closed intervals. Allowing such $\Sigma$s in Theorem \ref{dual}, we get this result. This conjecture is also supported by our numerical experiments.
\bibliographystyle{alpha}
\bibliography{main}

\begin{thebibliography}{{Smi}21}

\bibitem[AP08]{MR2428512}
Juli\'{a}n Aguirre and Juan~Carlos Peral.
\newblock The trace problem for totally positive algebraic integers.
\newblock In {\em Number theory and polynomials}, volume 352 of {\em London Math. Soc. Lecture Note Ser.}, pages 1--19. Cambridge Univ. Press, Cambridge, 2008.
\newblock With an appendix by Jean-Pierre Serre.

\bibitem[FS55]{MR72941}
M.~Fekete and G.~Szeg\"{o}.
\newblock On algebraic equations with integral coefficients whose roots belong to a given point set.
\newblock {\em Math. Z.}, 63:158--172, 1955.

\bibitem[OT23]{OT}
Bryce~Joseph {Orloski} and Naser {Talebizadeh Sardari}.
\newblock {Limiting distributions of conjugate algebraic integers}.
\newblock {\em arXiv e-prints}, page arXiv:2302.02872, February 2023.

\bibitem[OT24]{lowerbound}
Bryce~Joseph {Orloski} and Naser {Talebizadeh Sardari}.
\newblock {New Lower Bounds for the Schur-Siegel-Smyth Trace Problem}.
\newblock {\em arXiv e-prints}, page arXiv:2401.03252, January 2024.

\bibitem[Sch18]{Schur}
I.~Schur.
\newblock \"{U}ber die {V}erteilung der {W}urzeln bei gewissen algebraischen {G}leichungen mit ganzzahligen {K}oeffizienten.
\newblock {\em Math. Z.}, 1(4):377--402, 1918.

\bibitem[Ser19]{MR4093205}
Jean-Pierre Serre.
\newblock Distribution asymptotique des {V}aleurs {P}ropres des {E}ndomorphismes de {F}robenius [d'apr\`es {A}bel, {C}hebyshev, {R}obinson,{$\ldots$}].
\newblock {\em Ast\'{e}risque}, (414, S\'{e}minaire Bourbaki. Vol. 2017/2018. Expos\'{e}s 1136--1150):Exp. No. 1146, 379--426, 2019.

\bibitem[Ser20]{serre_curves}
Jean-Pierre Serre.
\newblock {\em Rational points on curves over finite fields}, volume~18 of {\em Documents Math\'{e}matiques (Paris) [Mathematical Documents (Paris)]}.
\newblock Soci\'{e}t\'{e} Math\'{e}matique de France, Paris, [2020] \copyright 2020.
\newblock With contributions by Everett Howe, Joseph Oesterl\'{e} and Christophe Ritzenthaler, Edited by Alp Bassa, Elisa Lorenzo Garc\'{\i}a, Christophe Ritzenthaler and Ren\'{e} Schoof.

\bibitem[Sie45]{MR12092}
Carl~Ludwig Siegel.
\newblock The trace of totally positive and real algebraic integers.
\newblock {\em Ann. of Math. (2)}, 46:302--312, 1945.

\bibitem[{Smi}21]{Smith}
Alexander {Smith}.
\newblock {Algebraic integers with conjugates in a prescribed distribution}.
\newblock {\em arXiv e-prints}, page arXiv:2111.12660, November 2021.

\bibitem[Smy84]{MR0762691}
Christopher Smyth.
\newblock Totally positive algebraic integers of small trace.
\newblock {\em Ann. Inst. Fourier (Grenoble)}, 34(3):1--28, 1984.

\bibitem[ST97]{logpotentials}
Edward~B. Saff and Vilmos Totik.
\newblock {\em Logarithmic potentials with external fields}, volume 316 of {\em Grundlehren der mathematischen Wissenschaften [Fundamental Principles of Mathematical Sciences]}.
\newblock Springer-Verlag, Berlin, 1997.
\newblock Appendix B by Thomas Bloom.

\bibitem[TV08]{Tao}
Terence Tao and Van Vu.
\newblock John-type theorems for generalized arithmetic progressions and iterated sumsets.
\newblock {\em Adv. Math.}, 219(2):428--449, 2008.

\end{thebibliography}
\end{document}